\documentclass[12pt]{amsart}
  \usepackage{latexsym} 
  \usepackage[all]{xy}
  \usepackage{amsfonts} 
  \usepackage{amsthm} 
  \usepackage{amsmath} 
  \usepackage{amssymb}
  \usepackage{pifont}  
  \usepackage{enumerate}
  \xyoption{2cell}
\usepackage{xcolor}
\usepackage{framed}
\definecolor{shadecolor}{gray}{0.8}
\definecolor{lgray}{gray}{0.5}
 \usepackage{marginnote}
\newcommand{\cC}{{C}}

 \usepackage{tikz}
\usetikzlibrary{arrows}

  \def\sw#1{{\sb{(#1)}}} 
   
  \def\su#1{{\sp{[#1]}}}

  \def\proofof#1{{\sl Proof~of~#1.~~}}

  \def\endproof{\hbox{$\sqcup$}\llap{\hbox{$\sqcap$}}\medskip} 
  \def\<{{\langle}} 
  \def\>{{\rangle}}

  \def\eps{\varepsilon}

  \def\note#1{{}} 
 \def\can{{\rm \textsf{can}}}

  \def\note#1{} 

  \def\cC{{\mathcal C}}

   \def\cH{{\mathcal H}}   
  
  \def\cO{{\mathcal O}}

  \def\beq{\begin{equation}} 
  \def\eeq{\end{equation}}

  \def\id{\mathrm{id}} 
  \def\im{{\rm Im}}

  \def\ot{{\otimes}}

  \def\roA{{\varrho^A}}

 \def\htau{\hat{\tau}}

 \def\coker{\mathrm{coker}}

    \def \hx{\widehat{x}}

\def \hy{\widehat{y}}
        \def \hz{\widehat{z}}

  \def\blambda{\bar{\lambda}}

  \newcounter{zlist} 
  \newenvironment{zlist}{\begin{list}{(\arabic{zlist})}{ 
  \usecounter{zlist}\leftmargin2.5em\labelwidth2em\labelsep0.5em 
  \topsep0.6ex
  \parsep0.3ex plus0.2ex minus0.1ex}}{\end{list}}

  \newcounter{blist} 
  \newenvironment{blist}{\begin{list}{(\alph{blist})}{ 
  \usecounter{blist}\leftmargin2.5em\labelwidth2em\labelsep0.5em 
  \topsep0.6ex 
  \parsep0.3ex plus0.2ex minus0.1ex}}{\end{list}} 

  \newcounter{rlist}

\def\stac#1{\raise-.2cm\hbox{$\stackrel{\displaystyle\otimes}{\scriptscriptstyle{#1}}$}}

\def\cten#1{\raise-.2cm\hbox{$\stackrel{\displaystyle\widehat{\otimes}}
{\scriptscriptstyle{#1}}$}}

  \def\Label#1{\label{#1}\ifmmode\llap{[#1] }\else 
  \marginpar{\smash{\hbox{\tiny [#1]}}}\fi} 
  \def\Label{\label}

  \newtheorem{proposition}{Proposition}[section]
  \newtheorem{lemma}[proposition]{Lemma} 
  \newtheorem{corollary}[proposition]{Corollary} 
  \newtheorem{theorem}[proposition]{Theorem} 

\theoremstyle{definition} 
  \newtheorem{definition}[proposition]{Definition}
    
  \newtheorem{example}[proposition]{Example}

  \theoremstyle{remark} 
  \newtheorem{remark}[proposition]{Remark}

  \newcounter{c} 
   
  \newcommand{\etyk}[1]{\vspace{-7.4mm}$$\begin{equation}\Label{#1} 
  \addtocounter{c}{1}} 
  \renewcommand{\]}{\ifnum \value{c}=1 $$\else \end{equation}\fi} 
\def\ot{\otimes}

\def\CC{{\mathbb C}}

\def\NN{{\mathbb N}}

\def\PP{{\mathbb P}}

\def\RR{{\mathbb R}}

\def\TT{{\mathbb T}}

\def\ZZ{{\mathbb Z}}

\def\gg{{\mathfrak G}}

\newcommand{\Cc}{\mathcal{C}}

\newcommand{\Ii}{\mathcal{I}}

\newcommand{\Ss}{\mathcal{S}}

\newcommand{\Vv}{\mathcal{V}}

\def\*C{{}^*\hspace*{-1pt}{\Cc}}

\def\text#1{{\rm {\rm #1}}}

 \def\1{\mathbf{1}}

      \def\tr{\mathrm{Tr}\hspace{3pt}}
\def\ha{\hat{a}}
\def\hb{\hat{b}}

\def\hx{\hat{x}}
\def\hy{\hat{y}}
\def\hz{\hat{z}}
\def\bpartial{\bar{\partial}}
\def\bdelta{\bar{\delta}}
\def\btau{\bar{\tau}}
\def\bomega{\bar{\omega}}
\def\pz{\partial_z}
\def\pzs{\partial_{z^*}}
\def\gk{\mathrm{GKdim}}
\def\gg{>\!\!>}
\def\pil{{\PP_\theta}}
\def\opil{{\cO(\PP_\theta)}}
\def\storus{{\cC^\infty(\TT_\theta)}}
\def\spil{{\cC^\infty(\PP_\theta)}}

\begin{document}
\title[On the noncommutative pillow, cones and lens spaces]{Smooth geometry of the noncommutative pillow, cones and lens spaces}

\author{Tomasz Brzezi\'nski}
 \address{ Department of Mathematics, Swansea University, 
  Swansea SA2 8PP, U.K.} 
  \email{T.Brzezinski@swansea.ac.uk}   
 \author[Andrzej Sitarz]{Andrzej Sitarz ${}^\dagger$}\thanks{${}^\dagger$ Supported by NCN grant 2012/06/M/ST1/00169}
 \address{Institute of Physics, Jagiellonian University,
prof.\ Stanis\l awa \L ojasiewicza 11, 30-348 Krak\'ow, Poland.\newline\indent
 Institute of Mathematics of the Polish Academy of Sciences,
\'Sniadeckich 8, 00-950 Warszawa, Poland.}
  \email{andrzej.sitarz@uj.edu.pl}   
 \date\today
 
\begin{abstract}
This paper proposes a new notion of  smoothness of algebras, termed {\em differential smoothness}, that  combines the existence of a top form in a differential calculus over an algebra together with a strong version of  the Poincar\'e duality realized as an isomorphism between complexes of differential and integral forms. The quantum two- and three-spheres, disc, plane and the noncommutative torus are all smooth in this sense. Noncommutative  coordinate algebras of deformations of several examples of classical orbifolds such as the pillow orbifold, singular cones and lens spaces are also differentially smooth.  Although surprising this is not fully unexpected as these algebras are known to be {\em homologically smooth}. The study of Riemannian aspects of the noncommutative pillow and Moyal deformations of cones  leads to spectral triples that satisfy the orientability condition that is known to be broken for classical orbifolds.
\end{abstract}
\maketitle

\tableofcontents

\section{Introduction}
It is often the case that a deformation of a singular variety or an orbifold produces a noncommutative object that behaves as if it were a smooth manifold. This observation underlies the theory of noncommutative resolutions \cite{Van:cre}. Recent papers \cite{Brz:smo} and \cite{Brz:con} contain illustrations of this phenomenon on the algebraic level on a number of (families of) explicit examples such as  quantum teardrops \cite{BrzFai:tea}, quantum (classically singular) lens spaces \cite{HonSzy:len}, the noncommutative pillow \cite{BraEll:sph} (see also \cite{Eva:orb} and \cite[Section~3.7]{EvaKaw:sym}) and quantum cones.  In \cite{Brz:smo} and \cite{Brz:con}  the smoothness is understood in the homological sense, i.e.\hspace{3pt}as  the existence of a finite-length resolution of algebras by finitely generated projective bimodules; see \cite[Erratum]{Van:rel}, \cite{Kra:Hoc}. In the present article we establish that  also from the point of  differential and spectral geometry the noncommutative pillow, quantum cones and lens spaces behave as smooth objects. 

The idea behind the {\em differential smoothness} of algebras is rooted in the observation that a classical smooth orientable manifold, in addition to de Rham complex of differential forms, admits also the complex of {\em integral forms} isomorphic to  the de Rham complex; \cite[Section~4.5]{Man:gau}. The de Rham differential can be understood as a special left connection, while the boundary operator in the complex of integral forms is an example of a {\em right connection}. In the standard (commutative) differential geometry a knowledge of the integral forms and right connections does not contribute to a better understanding of the structure of a manifold, the existence of the Hodge star and the Poincar\'e duality are fully sufficient. On the other hand it becomes  useful in defining the Berezin integral on a supermanifold, and  precisely in this context right connections and integral forms have been  introduced in \cite[Chapter~4]{Man:gau}. Supergeometry might be interpreted as one of the predecessors or  special cases of noncommutative geometry thus it seems quite natural to expect that the (rather mild) usefulness of integral forms in supergeometry should become more pronounced in the noncommutative setup. This expectation  led to the introduction of a noncommutative version of a right connection termed a {\em hom-connection} (or {\em divergence}) in \cite{Brz:con} and then to the associated complex of integral forms in \cite{BrzElK:int}. 

While every algebra admits a differential calculus (albeit not necessarily of any geometric interest), a priori not every algebra must admit a complex of integral forms relative to a given differential calculus. The main results of \cite{BrzElK:int} show that divergences (though not necessarily flat) can be defined for calculi generated by twisted multiderivations and the examples studied there feature complexes of integral forms isomorphic to noncommutative de Rham complexes (of dimension equal to the classical dimension of non-deformed spaces). This can be thought of as a strong version of Poincar\'e duality, since  it is an isomorphism of the complexes of differential and integrable forms, not just of their homologies, that is observed. Motivated by these examples, we introduce here the term {\em integrable differential calculus} to indicate a calculus that is isomorphic to the associated complex of integral forms.  Informally, an algebra that is a deformation of the coordinate algebra of a classical variety and has an integrable differential calculus of classical dimension can be understood as being {\em differentially smooth}. More formally, an affine or finitely generated algebra with integer Gelfand-Kirillov dimension, say $n$, is said to be {\em differentially smooth} if it admits an integrable $n$-dimensional differential calculus that is connected in the sense that the kernel of the differential restricted to the algebra consists only of scalar multiples of the identity. In this sense the algebras studied in \cite{BrzElK:int} are smooth, a fact that should not be too surprising, since these are  $q$-deformations of classically smooth compact manifolds (the two- and three-spheres and planes). The novelty of the present paper is that the differential smoothness is established for algebras that describe noncommutative versions of classically singular spaces (orbifolds).  In particular we prove:\medskip

\noindent{\bf Theorem~A.} 
{\em Coordinate algebras of the noncommutative pillow, cones and lens spaces are differentially smooth, i.e.\hspace{3pt}they admit connected integrable differential calculi of dimensions 2, 2 and 3 respectively.}
\medskip

Theorem~A indicates that noncommutative coordinate algebras of deformations of some classically non-smooth spaces or orbifolds are not only homologically smooth, but also admit types of differential calculi characteristic of smooth manifolds. In short, at this level at least, we are not able to distinguish between noncommutative deformations of manifolds and orbifolds, and should such a distinction be possible it must occur on finer geometric levels. The first such level that should be studied are the Riemannian or metric aspects of differentially smooth algebras, which following \cite{Connes} are encoded in spectral triples. It was shown in \cite{ReVa} that on the classical level, a spectral triple  can distinguish between a manifold and a (good) orbifold in the sense that the {\em orientability condition} (meaning the existence of a Hochschild cycle whose image in the differential calculus induced by the Dirac operator is the chirality operator) that holds in the former case fails to hold in the latter. We put the noncommutative pillow and a special case of the quantum cone corresponding to the Moyal deformation of the unit disc to the test, and find spectral triples that satisfy a milder version of the orientability condition, more precisely we construct cycles but not Hochschild cycles that map to the chirality operator. One might speculate whether the fact that constructed cycles are not Hochschild cycles is a remnant of the singular nature of the classical counterparts of noncommutative manifolds. The construction of spectral triples on any invariant subalgebras with respect to the action of a finite group has been studied in several cases following the standard method of restriction of the known spectral triple over the full algebra. This has been studied for the lens spaces \cite{SiVe} and three-dimensional Bieberbach manifolds \cite{Adalgott} allowing the classification of spin structures in the noncommutative counterparts of these manifolds. However, in neither of the constructed examples an orientation cycle was constructed. Although in the $q$-deformed case the existence of such a cycle could be doubtful (as the spectral triple for the $SU_q(2)$ itself has no orientation cycle) it is expected that such a cycle  exists for the more regular case of noncommutative quotients of the noncommutative torus and $\theta$-deformations. The  construction of an orientation cycle presented here for the noncommutative pillow and the Moyal $N=2$ cone is the first such result.  \medskip

\noindent {\bf Notation.}
All algebras considered in this paper are associative and unital over the complex field $\CC$. If $X$ is  a classical geometric space (affine space, manifold, orbifold etc.), $\cO(X_q)$ denotes the noncommutative coordinate algebra of the (non-existing in a usual sense) quantum space $X_q$. Although we write $\Omega A$ for a differential calculus over an algebra $A$, in order to avoid overloading notation, we write $\Omega(X_q)$ for a differential calculus over $\cO(X_q)$. Similarly, $\Ii_k (X_q)$ means integral $k$-forms over $\cO(X_q)$, i.e.\hspace{3pt}right $\cO(X_q)$-module homomorphisms $\Omega^k(X_q) \to \cO(X_q)$.

\section{Integrable differential calculi over noncommutative algebras}
\setcounter{equation}{0}
\subsection{Integrable calculi and differential smoothness of affine algebras}
By a {\em differential graded algebra} we mean a non-negatively graded algebra $\Omega$ (with the product traditionally denoted by   $\wedge$) together with the degree-one linear map $d: \Omega^\bullet \to \Omega^{\bullet+1}$ that satisfies the graded Leibniz rule and is such that $d\circ d=0$. We say that a differential graded algebra $(\Omega, d)$ is a {\em calculus over an algebra $A$} if $\Omega^0=A$ and, for all $n\in \NN$,  $\Omega^n=AdA\wedge dA\wedge \cdots \wedge dA$ ($dA$ appears $n$-times). In this case we write $\Omega A$. Note that due to the Leibniz rule $\Omega^n A = dA\wedge dA\wedge \cdots \wedge dA\, A$ too. A differential calculus $\Omega A$ is said to be {\em connected} if $\ker d\mid_A = \CC.1$. If $A$ is a complex $*$-algebra then it is often requested that $\Omega A$ be a $*$-algebra and that $*\circ d = d\circ *$. In this situation one refers to $\Omega A$ as  a $*$-differential calculus. If $B$ is a subalgebra of $A$, then by restriction of $\Omega A$ to $B$ we mean the differential calculus $\Omega B$ with $\Omega^n B = BdB\wedge dB\wedge \cdots \wedge dB \subseteq \Omega^nA$

A calculus $\Omega A$ is said to have {\em dimension $n$} if $\Omega^nA \neq 0$ and $\Omega^mA = 0$ for all $m >n$. An $n$-dimensional calculus $\Omega A$ admits a {\em volume form} if $\Omega^nA$ is isomorphic to $A$ as a left and right $A$-module. The existence of a right $A$-module isomorphism means that there is a free generator of $\Omega^nA$ (as a right $A$-module), i.e.\hspace{3pt}$\omega \in \Omega^nA$, such that all elements of $\Omega^nA$ can be uniquely written as $\omega a$, $a\in A$. We refer to such  a generator $\omega$ as to a {\em volume form} on $\Omega A$. 
The right $A$-module isomorphism $\Omega^n A \to A$ corresponding to a volume form $\omega$ is denoted by $\pi_\omega$, i.e.
\begin{equation}\label{vol.iso}
\pi_\omega(\omega a) = a, \qquad \mbox{for all $a\in A$}.
\end{equation}
Since $\Omega^n A$ is also isomorphic to $A$ as a left $A$-module, any volume form induces an algebra automorphism $\nu_\omega$ of $A$ by the formula 
\begin{equation}\label{vol.auto}
a \omega = \omega \nu_\omega(a).
\end{equation}

Dually to a differential calculus on $A$ one considers its {\em integral calculus}; see \cite{Brz:con}, \cite{BrzElK:int}. Let $\Omega A$ be a differential calculus on $A$. The space of $n$-forms $\Omega^nA$ is an $A$-bimodule. Let $\Ii_nA$ denote the right dual of $\Omega^n A$, i.e.\hspace{3pt}the space of all right $A$-linear maps $\Omega^n A \to A$. Each of the $\Ii_n A$ is an $A$-bimodule with the actions
$$
(a \cdot \phi \cdot b) (\omega) = a\phi(b\omega), \qquad \mbox{for all}\quad  \phi\in \Ii_{n} A,\hspace{3pt}\omega\in \Omega^n A,\hspace{3pt} a,b \in A.
$$
The direct sum  of all the $\Ii_nA$, denoted $\Ii A = \oplus_n \Ii_nA$, is a right $\Omega A$-module with action
\begin{equation}\label{int.right}
(\phi \cdot \omega)(\omega ') = \phi(\omega\wedge \omega'), \qquad \mbox{for all}\quad  \phi\in \Ii_{n+m} A,\hspace{3pt}\omega\in \Omega^n A,\hspace{3pt}\omega'\in \Omega^mA.
\end{equation}
A  {\em divergence} on $A$ is a linear map $\nabla: \Ii_1A\to A$, such that
\begin{equation}\label{hom.Leibniz}
\nabla (\phi\cdot a) = \nabla(\phi) a + \phi(da), \qquad \mbox{for all}\quad  \phi\in \Ii_{1} A,\hspace{3pt}a \in A.
\end{equation}
A divergence can be extended to the whole of $\Ii A$,  $\nabla_n: \Ii_{n+1} A\to \Ii_n A$, by setting 
\begin{equation}\label{hom.ext}
\nabla_n(\phi)(\omega)= \nabla  (\phi\cdot \omega) + (-1)^{n+1} \phi(d\omega), \qquad \mbox{for all \,} \phi\in \Ii_{n+1} A, \hspace{3pt}\omega \in \Omega^{n}A.
\end{equation}
A combination of \eqref{hom.Leibniz} with \eqref{hom.ext} yields the following Leibniz rule, for all $\phi\in \Ii_{m+n+1} A$ and $\omega\in \Omega^m A$,
\begin{equation}\label{hom.Leibniz.n}
\nabla_{n}(\phi \cdot \omega) = \nabla_{m+n}(\phi )\cdot \omega + (-1)^{m+n} \phi \cdot d\omega;
\end{equation}
see \cite[3.2~Lemma]{Brz:con}.

A divergence is said to be {\em flat} if $\nabla\circ \nabla_1 =0$. This then implies that $\nabla_n\circ\nabla_{n+1} =0$, for all $n\in \NN$, hence $\Ii A$ together with the $\nabla_n$ form a chain complex, which is termed the {\em complex of integral forms} over  $A$. The cokernel map of $\nabla$, i.e.\hspace{3pt}$\Lambda: A \to \coker \nabla = A/\im \nabla$ is called the {\em integral on $A$ associated to $\Ii A$}. Note that, in general, it is not guaranteed that a given differential calculus on $A$ will admit a divergence on $A$ and even if it admits such a divergence that it would be flat; see \cite{BrzElK:int}.

Given a left $A$-module  $X$ with action $a\cdot x$, for all $a\in A$, $x\in X$, and an algebra automorphism $\nu$ of $A$, the notation ${}^\nu X$ stands for $X$ with the $A$-module structure twisted by $\nu$, i.e.\hspace{3pt}with the $A$-action $a\otimes x\mapsto \nu(a)\cdot x$.

The following  definition introduces notions which form the backbone of the majority of this paper. 

\begin{definition}\label{def.integrable}
An $n$-dimensional  differential calculus $\Omega A$ is said to be {\em integrable} if 
 $\Omega A$ admits a complex of integral forms $(\Ii A, \nabla)$ for which 
there exist an algebra automorphism $\nu$ of $A$ and $A$-bimodule isomorphisms $\Theta_k: \Omega^k A \to {}^\nu \Ii_{n-k} A$, $k=0,\ldots , n$,  rendering commutative the following diagram:
$$
\xymatrix{ A\ar[r]^d \ar[d]_{\Theta_0} & \Omega^1A\ar[r]^d \ar[d]_{\Theta_1} & \Omega^2A \ar[r]^-d\ar[d]_{\Theta_2} & \ldots  \ar[r]^-d & \Omega^{n-1}A \ar[r]^d\ar[d]_{\Theta_{n-1}} &\Omega^nA\ar[d]^{\Theta_n}\\
{}^\nu \Ii_n A\ar[r]^-{\nabla_{n-1}} &{}^\nu \Ii_{n-1} A\ar[r]^-{\nabla_{n-2}} &  {}^\nu \Ii_{n-2} A\ar[r]^-{\nabla_{n-3}} & \ldots  \ar[r]^-{\nabla_1} &\hspace{3pt}{}^\nu \Ii_1 A\ar[r]^\nabla & {}^\nu \! A  \, .}
$$

The $n$-form $\omega: = \Theta_n^{-1}(1) \in \Omega^n A$ is called an {\em integrating volume form}.
\end{definition}

Examples of algebras admitting integrable calculi discussed in \cite{BrzElK:int} include the algebra of complex matrices $M_N(\CC)$ with the $N$-dimensional calculus generated by derivations \cite{Dub:der}, \cite{DubKer:non}, the quantum group $SU_q(2)$ with the three-dimensional left covariant calculus \cite{Wor:twi} and the quantum standard sphere with the restriction of the above  calculus. 

The following theorem indicates that the integrability of a differential calculus  can be defined without explicit reference to integral forms.

\begin{theorem}\label{thm.integrable}
The following statements about an $n$-dimensional differential calculus $\Omega A$ over an algebra $A$ are equivalent:
\begin{zlist}
\item $\Omega A$  is an integrable differential calculus.
\item There exist an algebra automorphism $\nu$ of $A$ and $A$-bimodule isomorphisms $\Theta_k: \Omega^k A \to {}^\nu \Ii_{n-k} A$, $k=0,\ldots , n$, such that, for all $\omega' \in \Omega^k A$, $\omega'' \in \Omega^m A$,
\begin{equation}\label{linearity}
\Theta_{k+m} (\omega'\wedge \omega'') = (-1)^{(n-1)m}\Theta_k(\omega')\cdot \omega''.
\end{equation}
\item There exist an algebra automorphism $\nu$ of $A$ and an $A$-bimodule map $\vartheta : \Omega^n A \to {}^\nu A$ such that all left multiplication maps 
$$
\ell^k_\vartheta : \Omega^kA \to \Ii_{n-k} A, \qquad \omega'\mapsto \vartheta \cdot \omega', \qquad k=0,1,\ldots , n,
$$
where the actions $\cdot$ are defined by \eqref{int.right}, are bijective.
\item $\Omega A$ admits a volume form $\omega$ such that 
all left multiplication maps 
$$
\ell^k_{\pi_\omega} : \Omega^kA \to \Ii_{n-k} A, \qquad \omega'\mapsto \pi_\omega \cdot \omega', \qquad k=1,\ldots , n-1,
$$
where $\pi_\omega$ is defined by \eqref{vol.iso}, are bijective.
\end{zlist}
\end{theorem}
\begin{proof}
(1) $\Rightarrow$ (2) 
The existence of an algebra automorphism $\nu$  and maps $\Theta_k: \Omega^k A \to {}^\nu \Ii_{n-k} A$ that make the diagram in Definition~
\ref{def.integrable} commute are parts of the definition of integrability of a differential calculus. We will prove that the  $\Theta_{k+m}$ satisfy equations \eqref{linearity} by induction with respect to $k+n$. 

First, for $k+m=0$ (\ref{linearity}) holds by definition. With the inductive
assumption that the formula is true for all $k+m < p \leq n$,  it needs to be demonstrated that (\ref{linearity}) is also true for $p$. Using the commutativity of the diagram in Definition~\ref{def.integrable}, the Leibniz rule and (\ref{hom.Leibniz.n}), we can compute: 
\begin{eqnarray*}
\Theta_{k+m+1} (\omega \wedge da) &=& (-1)^{k+m} \left( \Theta_{k+m+1}(d (\omega a) ) - \Theta_{k+m+1}(( d \omega) a) \right) \\
&=&  (-1)^{k+m} \left( \nabla_{n-k-m-1} \Theta_{k+m} (\omega a) -  \Theta_{k+m+1}( d \omega) a \right) \\
&=&  (-1)^{k+m} \left( \nabla_{n-k-m-1} ( \Theta_{k+m} (\omega) a) -  ( \nabla_{n-k-m-1} \Theta_{k+m} (\omega)) a \right) \\
&=&  (-1)^{k+m} (-1)^{n-k-m-1} \Theta_{k+m} (\omega) \cdot da = (-1)^{n-1}  \Theta_{k+m} (\omega) \cdot da.
\end{eqnarray*}
The inductive assumption and the fact that every element of $\Omega^{m} A$ is a linear combination of products of 
$m-1$-forms with exact one-forms, imply the required equality,
$$
\Theta_{k+m+1} (\omega'\wedge \omega'') = (-1)^{(n-1)(m+1)}\Theta_{k}(\omega')\cdot \omega'',
$$
for all $\omega' \in \Omega^k A$, $\omega'' \in \Omega^{m+1} A$. The assertion follows by the principle of mathematical induction.

(2) $\Rightarrow$ (3) Given a system of $A$-bimodule isomorphisms $\Theta_k: \Omega^k A \to {}^\nu \Ii_{n-k} A$, $k=0,\ldots , n$ that satisfy \eqref{linearity}, define
$$
\vartheta := \Theta_0(1).
$$
Since $\Theta_n$  satisfies \eqref{linearity}, for all $\omega\in \Omega^nA$,
$$
\vartheta (\omega) = \Theta_0(1)(\omega) = \Theta_0(1) \cdot \omega = \Theta_n(\omega).
$$
Thus $\vartheta = \Theta_n$ and, in particular, it is an $A$-bimodule map as stated. Again by \eqref{linearity}, for all $\omega'\in \Omega^kA$,
$$
\Theta_k (\omega') = (-1)^{(n-1)k} \Theta_0(1)\cdot \omega' = (-1)^{(n-1)k} \vartheta \cdot \omega' ,
$$
i.e.\hspace{3pt}$\Theta_k  = (-1)^{(n-1)k} \ell^k_\vartheta$, and hence all the $\ell^k_\vartheta$ are bijective, as required.

(3) $\Rightarrow$ (4) Note that by the definition of the action \eqref{int.right} $\vartheta = \ell_\vartheta^n$, hence $\vartheta$ is an $A$-bimodule isomorphism. Consequently, $\Omega^n A$ is isomorphic to $A$ as a left and right $A$-module, hence $\omega := \vartheta^{-1}(1)$ is a volume form. Since  $\vartheta^{-1}$ is a bimodule map  ${}^\nu A \to \Omega^n A$, for all $a\in A$,
$$
a = \pi_\omega(\omega a) = \pi_\omega(\vartheta^{-1}(1) a) = \pi_\omega (\vartheta^{-1}(a)),
$$
i.e.\hspace{3pt}$\pi_\omega = \vartheta$, hence all the $\ell^k_{\pi_\omega}$ are bijective.

(4) $\Rightarrow$ (1) Given a volume form $\omega \in \Omega^nA$, define
\begin{equation}\label{theta.k}
\Theta_k = (-1)^{(n-1)k} \ell^k_{\pi_\omega}: \Omega^kA \to \Ii_{n-k}A, \qquad k=0,1,\ldots, n.
\end{equation}
By assumption $\Theta_k$ are bijective for $k=1,\ldots, n -1$. Note that $\Theta_n = \pi_\omega$, hence it is bijective too. We will next show that the map
$$
\Theta^{-1}_0: \Ii_n A \to A, \qquad \phi\mapsto \nu_\omega^{-1} (\phi(\omega)),
$$
where $\nu_\omega$ is the algebra automorphism associated to $\omega$ via \eqref{vol.auto}, is the inverse of $\Theta_0$. For all $a\in A$,
\begin{eqnarray*}
\Theta_0^{-1}\left(\Theta_0(a)\right) &=& \nu_\omega^{-1}\left(\Theta_0(a)(\omega)\right) =  \nu_\omega^{-1}\left(\pi_\omega(a\omega)\right)\\
&=& \nu_\omega^{-1}\left(\pi_\omega(\omega\nu_\omega(a))\right) = \nu_\omega^{-1}\left(\nu_\omega(a)\right) =a,
\end{eqnarray*}
by the definitions of $\nu_\omega$ and $\pi_\omega$. On the other hand, for all $\phi\in \Ii_n A$ and $a\in A$,
\begin{eqnarray*}
\Theta_0\circ \Theta_0^{-1}\left( \phi\right)(\omega a) &=& \Theta_0\left( \nu_\omega^{-1} (\phi(\omega))\right)(\omega a) = \pi_\omega\left(\nu_\omega^{-1} (\phi(\omega))\omega a\right)\\
&=& \pi_\omega\left(\omega\phi(\omega) a\right) = \phi(\omega) a = \phi(\omega a),
\end{eqnarray*}
by the definitions of $\nu_\omega$ and $\pi_\omega$ and the right $A$-linearity of $\phi$. Since $\Omega^n A$ is generated by $\omega$ this means that the composite map $\Theta_0\circ \Theta_0^{-1}$ is the identity. Hence $\Theta_0^{-1}$ is the inverse of $\Theta_0$, as claimed.

Directly from definition \eqref{theta.k}, $\Theta_k$ satisfy equations \eqref{linearity}. In particular, they are right $A$-module maps. Next note that, for all $a\in A$, $\omega''\in \Omega^n$,
\begin{equation}\label{iso.bil}
\pi_\omega(a \omega'') = \nu_\omega(a) \pi_\omega( \omega''), \qquad \pi_\omega(\omega''a)  = \pi_\omega(\omega'')a.
\end{equation}
Thus, for all $a\in A$,  $\omega'\in \Omega^kA$, and $\omega''\in \Omega^{n-k}A$,
\begin{eqnarray*}
\Theta_k(a\omega')(\omega'') &=& (-1)^{(n-1)k} \pi_\omega(a\omega'\wedge\omega'') \\
&=& (-1)^{(n-1)k} \nu_\omega(a) \pi_\omega(\omega'\wedge\omega'') = \nu_\omega(a) \Theta_k(\omega')(\omega''),
\end{eqnarray*}
by the first of equation \eqref{iso.bil}. Therefore, all $\Theta_k $ given by \eqref{theta.k} are bimodule isomorphisms $\Omega^kA \to {}^{\nu}\Ii_{n-k}A$, where  $\nu = \nu_\omega$. 

Let us define:
\begin{equation}\label{nabla.k}
\nabla_k = \Theta_{n-k}\circ d\circ \Theta_{n-k-1}^{-1} : \Ii_{k+1} A\to \Ii_{k} A.
\end{equation}
Obviously, the maps $\nabla_k$ will make the diagram in Definition~\ref{def.integrable} commute, and, since $d\circ d =0$ also $\nabla_{k-1}\circ \nabla_k =0$. We need to prove that $\nabla := \nabla_0$ is a divergence and that all the remaining $\nabla_k$ given by \eqref{nabla.k} extend $\nabla$ in the sense of equalities \eqref{hom.ext}. 

For all $a\in A$ and $\phi \in \Ii_1A$,
\begin{eqnarray*}
\nabla(\phi\cdot a) &=& \Theta_{n}\circ d\circ \Theta_{n-1}^{-1}(\phi\cdot a) = \Theta_{n}\left(d\left(\Theta_{n-1}^{-1}(\phi) a\right)\right)\\
&=& \Theta_{n}\left(d\left(\Theta_{n-1}^{-1}(\phi)\right) a\right) + (-1)^{n-1} \Theta_{n}\left(\Theta_{n-1}^{-1}(\phi) \wedge d a\right)\\
&=& \Theta_{n}\left(d\left(\Theta_{n-1}^{-1}(\phi)\right)\right) a + \phi\cdot da =  \nabla(\phi) a + \phi(da),
\end{eqnarray*}
where the second equality follows by the right $A$-linearity of $\Theta_{n-1}$, the third one is a consequence of the graded Leibniz rule, the fourth one follows by the right $A$-linearity of $\Theta_{n}$ and by equation \eqref{linearity}, and the final equality is simply the definition of $\nabla$ in \eqref{nabla.k} and the action \eqref{int.right}. This proves that $\nabla$ is a divergence.

Observe that setting $\omega' = \Theta_k^{-1}(\phi)$ in \eqref{linearity} and then applying $\Theta_{k+m}^{-1}$ one obtains, for all $\omega''\in \Omega^mA$ and $\phi\in \Ii_{n-k}A$,
$$
 \Theta_{k+m}^{-1} (\phi \cdot \omega'') = (-1)^{(n-1)m} \Theta_k^{-1}(\phi) \wedge \omega'' .
 $$
This can be used (in the second equality below) to prove \eqref{hom.ext}. For any $\phi \in  \Ii_{m+1}A$ and $\omega''\in \Omega^mA$, 
\begin{eqnarray*}
\nabla(\phi\cdot \omega'') &=& \Theta_{n}\circ d\circ \Theta_{n-1}^{-1}(\phi\cdot \omega'') = (-1)^{(n-1)m}\Theta_{n}\left( d\left( \Theta_{n-1-m}^{-1}(\phi) \wedge \omega''\right)\right)\\
&=& (-1)^{(n-1)m}\Theta_{n}\left( d\left( \Theta_{n-1-m}^{-1}(\phi)\right) \wedge \omega''\right)\\
&& + (-1)^{(m+1)n -1}\Theta_{n}\left( \Theta_{n-1-m}^{-1}(\phi) \wedge d \omega''\right)\\
&=& \Theta_{n-m}\left( d\left( \Theta_{n-1-m}^{-1}(\phi)\right) \right)\cdot \omega''
 + (-1)^{m+1}\Theta_{n-m-1}\left( \Theta_{n-1-m}^{-1}(\phi)\right) \wedge d \omega'' \\
 &=& \nabla_m(\phi)( \omega'' )+ (-1)^m \phi\cdot d\omega'',
\end{eqnarray*}
where the third equality follows by the Leibniz rule, the fourth one is a consequence of the identities \eqref{linearity} that all the $\Theta_k$ obey and the final equality follows by the definitions of $\nabla_m$ and the action \eqref{int.right}. This proves that each of the $\nabla_k$ defined by \eqref{nabla.k} is an extension of $\nabla$, and thus completes the proof of the theorem.
\end{proof}

\begin{remark}\label{rem.integrating}
A volume form $\omega \in \Omega^n A$ is an  integrating form if and only if it satisfies conditions (4) of Theorem~\ref{thm.integrable}.
\end{remark}

As it stands the integrability of a calculus over an algebra $A$ is a property of the differential graded algebra $\Omega A$ rather than a characterization of $A$ itself. To connect this property of $\Omega A$ to the nature of $A$ we need to relate the dimension of the differential calculus with that of $A$. Since we are dealing with algebras that are deformations of coordinate algebras of affine varieties, the Gelfand-Kirillov dimension seems to be best suited here; see \cite{KraLen:gro} or\cite[Chapter~8]{McCRob:Noe} for a detailed discussion of the Gelfand-Kirillov dimension.

Let $A$ be a finitely generated or affine algebra with generating subspace $\Vv$. Let us write $\Vv(n)$ for the subspace of $A$ spanned by 1 and all words in generators of $A$ of length at most $n$. The algebra $A$ is said to have {\em polynomial growth} if there exist $c\in \RR$ and $\nu\in \NN$ such that $\dim \Vv(n) \leq cn^\nu$ for all sufficiently large $n$. The {\em Gelfand-Kirillov dimension} of $A$ is a real number defined as
\begin{equation}\label{GK}
\gk(A) := \inf \{ \nu\; |\; \dim \Vv(n) \leq n^\nu, \, n \gg 0\},
\end{equation}
if $A$ has polynomial growth and is defined as infinity otherwise. In the case of commutative affine algebras with polynomial growth, the Gelfand-Kirillov dimension coincides with the dimension of the underlying affine space (the Krull dimension of its coordinate algebra). 

\begin{definition}\label{def.smooth.algebra}
An affine algebra with integer Gelfand-Kirillov dimension $n$ is said to be {\em differentially smooth} if it admits an $n$-dimensional connected integrable differential calculus.
\end{definition}

For example, the polynomial algebra $\CC[x_1,\ldots, x_n]$ has the Gelfand-Kirillov dimension $n$ and the usual exterior algebra is  an $n$-dimensional integrable calculus, hence $\CC[x_1,\ldots, x_n]$ is differentially smooth. The results of \cite{BrzElK:int} establish differential smoothness of  coordinate algebras of the quantum group $SU_q(2)$, the  standard quantum Podle\'s sphere and the quantum Manin plane. The following example shows that not all algebras are differentially smooth.

\begin{example}\label{ex.nonsmooth}
The algebra $A = \CC[x,y]/\langle xy\rangle$ is not differentially smooth.
\end{example}

\begin{proof}
Since $xy=yx = 0$, the algebra $A$ has a basis $1, x^n, y^n$, $n=1,2,\ldots$, hence $\gk(A) =1$. Suppose there is a one-dimensional 
connected integrable calculus $\Omega A$ and let $\Theta_1: \Omega^1 A \to {{}^\nu} \! A$,  where  $\nu$ is an algebra automorphism on $A$,  be the required bimodule isomorphism. The Leibniz rule together with the equalities  $xy=yx = 0$ imply that
$$ x\, dy = - dx \, y, \;\;\;\;\; y\, dx = - dy \, x. $$
Apply  $\Theta_1$ to these identities to obtain
\begin{equation}\label{nu.theta}
 \nu(x) \theta_y = - \theta_x y, \;\;\;\; \nu(y) \theta x = - \theta_y x, 
 \end{equation}
where $\theta_x := \Theta_1(dx)$ and $\theta_y :=\Theta_1(dy)$. 

The algebra $A$ admits two types of automorphisms 
$\nu_i:A\to A$, $i=1,2$,
$$
\nu_1(x) = ax, \quad \nu_1(y) = by, \;\;\;\; \hbox{and} \;\;\;\; \nu_2(x) = ay, \quad \nu_2(y) = bx,\qquad a,b\in \CC, \; a,b\neq 0.
$$
For $\nu=\nu_1$, \eqref{nu.theta} read
$$ a x \theta_y = - \theta_x y, \;\;\;\; b y \theta_x = - \theta_y x, $$
with the only solutions given by $ \theta_y = y p(y)$ and $\theta_x = x q(x)$, where $p,q$ are polynomials. 
Hence the image under $\Theta_1$ of any one-form must be a polynomial without a scalar term, so 
$1$ cannot lie in $\Theta_1( \Omega^1 A)$, and therefore $\Theta_1$ is not surjective.
In the case of  the automorphisms $\nu_2$, equations \eqref{nu.theta} come out as 
$$ a y \theta_y = - \theta_x y, \;\;\;\; b x \theta_x = - \theta_y x.$$
Since the algebra is commutative,
$$ y (a \theta_y + \theta_x) = 0 = x (b \theta_x + \theta_y), $$
which can be solved. If $ab \not= 1$ then again both $\theta_x$ and $\theta_y$ are polynomials without any scalar terms and,  by the same arguments as above, $\Theta_1$ is not surjective which contradicts
 the assumption that $\Theta_1$ is an isomorphism of bimodules. 
If $ab=1$, then the only solution is $\theta_y = c$ and $\theta_x = - ac$, for some $c \in \CC$. However,
in this case  $d (a y + x) = 0$,  which contradicts the assumption that the differential calculus 
is connected.

Therefore, there are no one-dimensional connected integrable calculi over $A$, i.e.\hspace{3pt}$A$ is not differentially smooth.
\end{proof}

\subsection{Finitely generated and projective integrable calculi}

Geometrically the most interesting cases of differential calculi are those where $\Omega^kA$ are finitely generated and projective right or left (or both) $A$-modules.
\begin{lemma}\label{lem.projective}
Let $\Omega A$ be an integrable $n$-dimensional calculus over $A$ with integrating form $\omega$. Then $\Omega^k A$ 
is a finitely generated projective right $A$-module if there exist a finite number of forms $\omega_i \in \Omega^kA$ and 
$\bar{\omega}_i \in \Omega^{n-k}A$ such that, for all $\omega'\in \Omega^k A$,
\begin{equation}
\omega'= \sum_i \omega_i \pi_\omega(\bar{\omega}_i \wedge \omega').
 \label{homproj}
\end{equation}
\end{lemma}
\begin{proof}
The elements $\omega_i$ are generators of $\Omega^kA$, while the equalities
$$
\phi_i (\omega')= \ell^{n-k}_{\pi_\omega}(\bar{\omega}_i) (\omega') = \pi_\omega(\bar{\omega}_i \wedge \omega'),
$$
define $\phi_i \in \hbox{Hom}_A(\Omega^kA, A)$, and then \eqref{homproj} guarantees that $\omega_i, \phi_i$ form a dual basis for $\Omega^kA$, hence  $\Omega^kA$ is projective.

To show the implication in the other direction let us assume that $\Omega^k A$ is finitely generated projective with a
basis $\omega_i$ and the dual basis $\phi_i$. Let us define:
$$ \bar{\omega}_i := \Theta_{n-k}^{-1} (\phi_i). $$
Using the properties of an integrable differential calculus it is easy to show that \eqref{homproj} is satisfied.
\end{proof}

\begin{lemma}\label{lem.integrating}
Let $\Omega A$ be an $n$-dimensional calculus over $A$ admitting a volume form $\omega$. Assume that, for all $k=1,2,\ldots, n-1$, there exist a finite number of forms $\omega_i^k, \bomega_i^k \in \Omega^kA$ such that, for all $\omega'\in \Omega^kA$,
\begin{equation}\label{dual.basis}
\omega' = \sum_i\omega_i^k \pi_\omega(\bomega_i^{n-k} \wedge \omega ') = \sum_i \nu_\omega^{-1}\left(\pi_\omega(\omega'\wedge  \omega_i^{n-k} )\right)\bomega_i^k,
\end{equation}
where $\pi_\omega$ and $\nu_\omega$ are defined by \eqref{vol.iso} and \eqref{vol.auto}, respectively. Then $\omega$ is an integrating form and  all the $\Omega^kA$ are finitely generated and projective as left and right $A$-modules.
\end{lemma}
\begin{proof}
Conditions \eqref{dual.basis} imply that $\omega^k_i$, $\pi_\omega(\bomega_i^{n-k} \wedge -)$ form a dual basis for $\Omega^k A$ as a right $A$-module and $\bomega^k_i$, $\nu_\omega^{-1}\left(\pi_\omega(-\wedge  \omega_i^{n-k} )\right)$ form a dual basis for $\Omega^k A$ as a left $A$-module.

For all $k$ define,
$$
\Phi_k : \Ii_{n-k} A \to \Omega^k A, \qquad \phi \mapsto \sum_i \nu_\omega^{-1}\left(\phi(\omega_i^{n-k})\right)\bomega_i^k.
$$
Then, for all $\phi \in \Ii_{n-k}$ and $\omega'\in \Omega^{n-k} A$,
\begin{eqnarray*}
\left(\ell_{\pi_\omega}^k \circ \Phi_k\right)(\phi)(\omega ') &=& \pi_\omega \left(\Phi_k(\phi)\wedge \omega '\right)\\
&=& \sum_i\pi_\omega \left(\nu_\omega^{-1}\left(\phi(\omega_i^{n-k})\right)\bomega_i^k\wedge \omega '\right)\\
&=& \sum_i\phi(\omega_i^{n-k})\pi_\omega \left(\bomega_i^k\wedge \omega '\right) = \sum_i\phi\left(\omega_i^{n-k}\pi_\omega \left(\bomega_i^k\wedge \omega '\right)\right)  = \phi(\omega'),
\end{eqnarray*}
where the first of equations \eqref{iso.bil} was used in the derivation of the third equality, and next the right $A$-linearity of $\phi$  and the first of equations \eqref{dual.basis} were employed. On the other hand, for all $\omega'\in \Omega^k A$,
$$
\left(\Phi_k\circ \ell_{\pi_\omega}^k  \right)(\omega ') = \sum_i\nu_\omega^{-1}\left(\ell_{\pi_\omega}^k(\omega ') (\omega_i^{n-k}) \right)\bomega_i^k\\
= \sum_i \nu_\omega^{-1}\left(\pi_\omega(\omega'\wedge  \omega_i^{n-k} )\right)\bomega_i^k  = \omega',
$$
by the second of equations \eqref{dual.basis}. Therefore, all of the $\ell_{\pi_\omega}^k$ are isomorphisms, and hence $\omega$ is an integrating form.
\end{proof}

In the set-up of Lemma~\ref{lem.integrating}, a combination of the form of the inverse of the $\ell_{\pi_\omega}^k$ together with that of the divergence associated to the isomorphisms $\ell_{\pi_\omega}^k$ via \eqref{nabla.k}, gives the following formula for the divergence:
\begin{equation}\label{hom.con.omega}
\nabla(\phi) = (-1)^{n-1} \sum_i \pi_\omega \left( d\left(\nu_\omega^{-1}\left(\phi(\omega_i^{1})\right)\right)\bomega_i^{n-1}\right),
\end{equation}
for all $\phi \in \Ii_1 A$.

\subsection{A construction of two-dimensional integrable calculi}

In this section we construct a class of two-dimensional integrable calculi. Notwithstanding its quite technical nature and  rather specialized appearance,  the following lemma is applicable in a variety of situations  including, of course, those that are the subject matter of this article.

\begin{lemma}\label{lemma.Poincare}
Assume that:
\begin{blist}
\item $A$ is an algebra and $B\subseteq A$ is a subalgebra, hence  $A$ is a $B$-bimodule in a natural way;
\item $A_+, A_- \subseteq A$ are right $B$-submodules of $A$ such that $A_+A_- = A_-A_+ = B$;
\item  there exists a two-dimensional differential calculus $\Omega B$ over $B$ with $\Omega^1 B = A_+\oplus A_-$ as a right $B$-module and the product in $\Omega^1 B$ given by the formula
$$
(a_+, a_-)\wedge (b_+, b_-) = \omega( \sigma_+(a_+)b_- + \sigma_-(a_-)b_+),
$$
for all $a_\pm,b_\pm \in A_\pm$, where   $\omega$ is a volume form and $\sigma_\pm: A_\pm\to A_\pm$  are invertible linear maps.
\end{blist}
Then $\Omega B$ is integrable.
\end{lemma}
Before we start proving Lemma~\ref{lemma.Poincare} let us point to the classical motivation behind it. In classical (complex) geometry, Riemann surfaces can be obtained as quotients of the disc (with hyperbolic metric) by Fuchsian groups. As a result, algebras of functions and  modules of (holomorphic or antiholomorphic) sections of the cotangent bundle over such a surface can be embedded in the algebra of functions on the disc and the corresponding modules over it. In particular, (anti-)holomorphic sections over a Riemann surface can be expressed in terms of functions on the disc. In the classical situation $B$ should be thought of as functions on a Riemann surface, $A$ as functions on the disc and $A_\pm$ as (anti-)holomorphic sections on the surface expressed in terms of functions on the disc. 

\proofof{Lemma~\ref{lemma.Poincare}}
In view of Theorem~\ref{thm.integrable} we only need to show that the map
$$
\Theta := \ell_\omega^k: \Omega^1 B \to \Ii_1 B, \qquad  \omega' \mapsto [\omega'' \mapsto \pi_\omega(\omega'\wedge \omega'')],
$$
is bijective.
Since $A_\pm A_\mp = B$, there exist $r_\pm ^i,s_\pm ^i\in A_\pm$ such that
$$
 r_+^i r_-^i = s_-^i s_+^i =1,
$$
where  here and below the repeated index is summed. Let us define a linear map:
\begin{equation}\label{def.theta}
\Theta^{-1}: \Ii_1 B\to \Omega^1 B, \qquad \phi \mapsto \left(\sigma_+^{-1}\left( \phi(0,s_-^i)s_+^i\right), \sigma_-^{-1}\left( \phi(r_+^i,0)r_-^i\right)\right).
\end{equation}
Then $\Theta^{-1}$ is the inverse of  $\Theta$. Indeed, for all $(a_+,a_-)\in \Omega^1 B$,
\begin{eqnarray*}
\Theta^{-1}\circ\Theta (a_+,a_-) &=& \left(\sigma_+^{-1}\left(\pi_\omega\left((a_+,a_-)\wedge (0,s_-^i)\right)s_+^i\right), \sigma_-^{-1}\left(\pi_\omega\left((a_+,a_-)\wedge (r_+^i,0)\right)r_-^i\right)\right)\\
&=& \left(\sigma_+^{-1} \left(\sigma_+  \left( a_+\right)s_-^is_+^i\right), \sigma_-^{-1} \left(\sigma_-  \left( a_-\right)r_+^ir_-^i\right)\right) = (a_+,a_-),
\end{eqnarray*}
and, for all $\phi\in \Ii_1 B$,
\begin{eqnarray*}
\Theta\circ\Theta^{-1}(\phi)(a_+,a_-) &=& \pi_\omega\left(\sigma_+^{-1} \left( \phi(0,s_-^i)s_+^i\right), \sigma_-^{-1} \left( \phi(r_+^i,0)r_-^i\right)\wedge (a_+,a_-)\right)\\
&=& \sigma_+ \left(\sigma_+^{-1}\left( \phi(0,s_-^i)s_+^i\right)\right) a_- + \sigma_-\left(\sigma_-^{-1}\left( \phi(r_+^i,0)r_-^i\right)\right) a_+\\
&=& \phi\left(0,  s_-^is_+^ia_-\right) + \phi\left( r_+^ir_-^ia_+,0\right) = \phi(a_+,a_-),
\end{eqnarray*}
where we used that $s_+^ia_-, r_-^ia_+\in B$ and the fact that $\phi$ is a right $B$-linear map. Now Theorem~\ref{thm.integrable} implies that $\Omega B$ is an integrable differential calculus as claimed.
\endproof

A typical problem to which Lemma~\ref{lemma.Poincare} can be applied is the construction of an integrable differential calculus  over an invariant part of a strongly group-graded algebra. 
Let $G$ be an Abelian group. Recall that an algebra $A$ is called a {\em $G$-graded algebra} if $A= \oplus_{g\in G} A_g$ and, for all $g,h\in G$, $A_gA_h\subseteq A_{g+h}$, and it is said to be {\em strongly-graded} if  $A_gA_h =  A_{g+h}$. In the strongly-graded case, for all $h\in G$, $A_{-h}A_h =  A_hA_{-h} = A_0$, where $A_0$ is the invariant subalgebra, i.e.\hspace{3pt}the subalgebra of all elements of $A$ graded by the neutral element $0\in G$. In this case, one can choose $B=A_0$, $A_+=A_h$, $A_-=A_{-h}$ (for a suitable $h\in G$), and $\sigma_\pm$ to be restrictions of any degree-preserving automorphism of $A$.  For example, Lemma~\ref{lemma.Poincare} provides one with  a proof of  integrability of the two-dimensional differential calculus over the  standard quantum Podle\'s sphere, alternative to that given in \cite[Section~4]{BrzElK:int}. In this case $A$ is the coordinate algebra of $SU_q(2)$, which is strongly graded by the integer group $\ZZ$. The invariant part of $A$ is the coordinate algebra of the quantum standard Podle\'s sphere \cite{Pod:sph}, $h=1$,  and the degree preserving automorphism of $A$ is induced by  the 3D-calculus on $A$.   

\subsection{Integrability and principality}
Strongly graded algebras are examples of principal comodule algebras. Let $H$ be a Hopf algebra with bijective antipode. Recall from \cite{BrzHaj:Che} that a right $H$-comodule algebra $A$ with coaction $\roA: A\to A\ot H$ is called a {\em principal comodule algebra} if the canonical map
$$
\can : A\ot_B A\to A\ot H, \qquad a\ot a'\mapsto a\roA(a'),
$$ is bijective and there exists a right $B$-module and right $H$-comodule splitting of the multiplication map $B\ot A\to A$. Here  $B$ is the coinvariant subalgebra, $B = A^{coH} :=  \{b\in A\; |\; \roA(b) = b\ot 1\}$. Principal comodule algebras play the role of principal fibre bundles in noncommutative geometry and, classically, a quotient of a smooth manifold by a free action of a Lie group is smooth, thus it is natural to expect that if a principal comodule algebra admits an integrable calculus so does its coinvariant algebra. 
In this section we consider a special case of an integrable differential calculus over a principal comodule algebra that induces such a calculus over the coinvariant subalgebra.

Let $A$ be a right $H$-comodule algebra. A differential calculus $\Omega A$ is said to be {\em $H$-covariant} if $\Omega A$ is a right $H$-comodule algebra with the degree-zero coactions $\varrho^{\Omega^k A}: \Omega^kA\to \Omega^kA\ot H$ that commute with the differential, i.e.\hspace{3pt}such that 
$$
\varrho^{\Omega^{k+1} A}\circ d  = (d\ot \id)\circ \varrho^{\Omega^k A}. 
$$
If $B$ is a coinvariant subalgebra of $A$, then the covariant calculus $\Omega A$ restricts to the calculus $\Omega B$ on $B$. Clearly,  $\Omega B$  is contained in the coinvariant part of $\Omega A$, $\Omega B \subseteq \left(\Omega A\right)^{co H}$, but the converse inclusion is not necessarily true.

\begin{lemma}\label{lem.princ}
Let $A$ be a principal $H$-comodule algebra with the coinvariant subalgebra $B$. Let $\Omega A$ be an $H$-covariant $n$-dimensional integrable calculus with integrating form $\omega$. Assume that:
\begin{blist}
\item $\Omega B = \left(\Omega A\right)^{co H}$;
\item $\omega$ is invariant, i.e.\hspace{3pt}$\rho^{\Omega^n A} (\omega) = \omega \ot 1$, and hence $\omega \in \Omega^n B$;
\item for all $k=1,\ldots , n-1$, if $\omega' \in \Omega^kA$ has the property that, for all $\omega'' \in \Omega^{n-k} B$, $\omega'\wedge \omega'' \in \Omega^n B$, then $\omega' \in \Omega^kB$.
\end{blist}
Then $\Omega B$ is an integrable differential calculus and $\omega$ is its integrating form.
\end{lemma}
\begin{proof}
By assumption (b), $\Omega^n B \cong B$ with the isomorphism $\pi_\omega^B : \Omega^nB\to B$ which is the restriction of $\pi_\omega ^A: \Omega^nA\to A$, $\omega a\mapsto a$.

Let us write $\varrho^{\Omega^k A}(\omega ') = \omega'\sw 0 \ot \omega'\sw 1$ and $h\su 1 \ot h \su 2 = \can^{-1}(1\ot h)$, summation understood in both cases. The map $h\mapsto \can^{-1}(1\ot h)$ is known as the {\em translation map}. By the properties of the translation map (see \cite[3.4~Remark]{Sch:rep}), for all $\omega''\in {\Omega^{n-k} A}$, there is an inclusion
$$
\omega''\sw 0 \omega''\sw 1\su 1 \ot \omega''\sw 1\su 2 \in \left(\Omega^{n-k}A\right)^{co H} \ot_B A = \Omega^{n-k}B \ot_B A
$$
where the last equality is a consequence of assumption (a). Therefore, for all $\phi \in \Ii_{n-k} B$, one can define $\hat{\phi} \in \Ii_{n-k} A$ by
\begin{equation}\label{hat.phi}
 \hat{\phi}: \omega'' \mapsto \phi\left(\omega''\sw 0 \omega''\sw 1\su 1\right) \omega''\sw 1\su 2.
\end{equation}
 Note that $\hat{\phi}(\omega'') = \phi(\omega'')$, for all $\omega''\in \Omega^{n-k}B$. Further note that, for all $\omega'\in \Omega^k B$, 
 \begin{equation}\label{hat}
 \widehat{\ell^k_{\pi_\omega^B}(\omega ')} = \ell^k_{\pi_\omega^A}(\omega ').
 \end{equation}
 Indeed, take any $\omega''\in \Omega^{n-k} A$. Then
 \begin{eqnarray*}
  \widehat{\ell^k_{\pi_\omega^B}(\omega ')}(\omega'') &=& \pi_\omega^B\left(\omega'\wedge \omega''\sw 0 \omega''\sw 1\su 1\right) \omega''\sw 1\su 2 = \pi_\omega^A\left(\omega'\wedge \omega''\sw 0 \omega''\sw 1\su 1\right) \omega''\sw 1\su 2\\
  &=& \pi_\omega^A\left(\omega'\wedge \omega''\sw 0 \omega''\sw 1\su 1\omega''\sw 1\su 2\right)  = \pi_\omega^A\left(\omega'\wedge \omega''\right) = \ell^k_{\pi_\omega^A}(\omega ')(\omega ''),
 \end{eqnarray*}
 where the third equality follows by the right $A$-linearity of $\pi_\omega^A$ and the fourth one by the fact that $h\su 1 h\su 2 = \eps(h)$, where $\eps$ is the counit of $H$.

 Take any $\phi \in \Ii_{n-k} B$. By assumption, there exists a unique $\omega'\in \Omega^k A$, such that
 \begin{equation}\label{inverse.om}
 \hat{\phi} = \ell_{\pi_\omega^A }^k(\omega ').
\end{equation}
 Hence, for all $\omega'' \in \Omega^{n-k}B$, 
 \begin{equation}\label{phi.omega}
 \phi(\omega'') = \hat{\phi}(\omega'') = \pi_{\omega}^A (\omega' \wedge \omega''), 
 \end{equation}
 i.e.\hspace{3pt}$\omega' \wedge \omega'' = {\pi_{\omega}^A}^{-1}(\phi(\omega''))$. Since $\phi(\omega'') \in B$, $ {\pi_{\omega}^A}^{-1}(\phi(\omega''))\in \Omega^n B$, and assumption (c) implies that $\omega'\in \Omega^kB$. 
 
 For all $k=1,\ldots , n-1$ define the maps
 $$
 \Phi_k : \Ii_{n-k} B \to \Omega^k B, \qquad \phi \mapsto \omega',
 $$
 where $\omega'$ is given by \eqref{inverse.om}. We claim that $\Phi_k $ is the inverse of $\ell_{\pi_\omega^B }^k$. For any $\phi\in  \Ii_{n-k} B$,
 $$
 \ell_{\pi_\omega^B }^k\left(\Phi_k\left(\phi\right)\right) = \ell_{\pi_\omega^B }^k\left(\omega'\right) = \ell_{\pi_\omega^A }^k\left(\omega'\right)\mid_{\Omega^{n-k}B} = \hat{\phi} \mid_{\Omega^{n-k}B} = \phi.
 $$
 In the converse direction, for all $\omega'\in \Omega^{k} B$,
 $$
 \Phi_k\left( \ell_{\pi_\omega^B }^k\left(\omega'\right)\right) = \omega'' \in \Omega^kB,
 $$
 where 
 $$
 \ell_{\pi_\omega^A }^k\left(\omega''\right) = \widehat{\ell_{\pi_\omega^B }^k\left(\omega'\right)} = \ell_{\pi_\omega^A }^k\left(\omega'\right),
 $$
 by \eqref{hat}. Since the maps $ \ell_{\pi_\omega^A }^k$ are bijective, $\omega'' = \omega'$, as  required.
\end{proof}

\begin{remark}
The assumptions of Lemma~\ref{lem.princ} mean that both $A$ and $B$ have integrable calculi of equal dimensions. 
Geometrically this limits the applicability of Lemma~\ref{lem.princ} to actions of finite quantum groups.
\end{remark}

In view of the construction of the isomorphisms $\Phi_k$ in Lemma~\ref{lem.princ}, the divergence $\nabla_B : \Ii_1\to B$ corresponding to the integrable calculus $\Omega B$ is related to the divergence $\nabla_A: \Ii_1\to A$ by
\begin{equation}\label{nablas}
\nabla_B(\phi) = \nabla_A(\hat{\phi}), \qquad \mbox{for all}\quad \phi \in \Ii_1 B.
\end{equation}
In particular this means that the image of $\nabla_B$ is contained in the image of $\nabla_A$, hence it is contained in the kernel of the integral $\Lambda_A : A\to \coker \nabla_A$. By the universal property of cokernels, there must exist a unique map $\chi: \coker \nabla_B \to \nabla_A$ connecting the integrals on $A$ and $B$ by
\begin{equation}\label{integrals}
\Lambda_A\mid_B = \chi\circ\Lambda_B
\end{equation}
In favourable situations, the formula \eqref{integrals} allows one to integrate on $B$ using integration on $A$.

\section{The noncommutative pillow}
\setcounter{equation}{0}
In this section we study the noncommutative version of one of the prime examples of a {\em good orbifold}, (meaning a singular space obtained by a non-free action of a finite group on a smooth manifold) known as a {\em pillow orbifold}  and obtained as a quotient of the two-torus by an action of the cyclic group $\ZZ_2$.  \cite[Chapter~13]{Thu:geo}. 

\subsection{Two-dimensional integrable differential calculus over the noncommutative pillow}\label{sec.dif}
The coordinate algebra of the noncommutative torus, $\cO(\TT^2_\theta)$, is a complex $*$-algebra generated by unitary $V, W$ subject to the relation
\begin{equation}\label{torus}
VW = \lambda WV, \qquad \mbox{where}\quad \lambda=\exp(2\pi i\theta).
\end{equation} 
We assume that $\theta$ is irrational. The algebra map 
\begin{equation} \label{sigma.auto}
\sigma: \cO(\TT^2_\theta)\to \cO(\TT^2_\theta), \qquad V \mapsto V^* \quad \mbox{and} \quad W\mapsto  W^*,
\end{equation}
 is an involutive automorphism, and hence it splits $\cO(\TT^2_\theta)$ into a direct sum $\cO(\TT^2_\theta)=\cO(\TT^2_\theta)_+\oplus \cO(\TT^2_\theta)_-$, where $a\in \cO(\TT^2_\theta)_\pm$ if and only if $\sigma(a) = \pm a$. This splitting means also that $\cO(\TT^2_\theta)$ is a $\ZZ_2$-graded algebra. The fixed point subalgebra $\opil:= \cO(\TT^2_\theta)_+$ is the coordinate algebra of the {\em noncommutative pillow orbifold} introduced in \cite{BraEll:sph}. It is generated by $x= V+V^*$ and $y=W+W^*$, but it is also convenient to consider $z=VW^* + V^*W$, which can be expressed in terms of $x$ and $y$ via 
\begin{equation}\label{eq.z}
(1-\lambda^2)z = xy - \lambda yx.
\end{equation}
As a vector space $\opil$ is spanned by $1$ and 
\begin{equation}\label{ab}
a_{m,n} = V^mW^n +V^{*m}W^{*n} \quad  b_{m,n} = V^mW^{*n} +V^{*m}W^{n}, \quad m,n\in \NN,\; m+n>0.
\end{equation}
$\opil$ can also be identified with an algebra generated by self-adjoint $x,y,z$ subject to relations \eqref{eq.z} and
\begin{subequations}\label{xy.radial}
\begin{equation}\label{eq.xy}
xz-\bar{\lambda}zx = (1-\bar{\lambda}^2) y, \qquad zy-\bar{\lambda}yz = (1-\bar{\lambda}^2) x,  
\end{equation}
\begin{equation}\label{radial}
x^2 + y^2 +\bar{\lambda}z^2 -xzy =2(1+\bar{\lambda}^2).
\end{equation}
\end{subequations}
It is equation \eqref{radial} that allows one to interpret $\pil$ as a deformation of the pillow orbifold. 
Using the relations (\ref{eq.z})--(\ref{xy.radial}) we can write a linear basis for $\opil$ in the generators $x,y$ as 
$x^k y^l$ and $x^k (y x) y^l$, $k,l\in \NN$. Therefore, the subspace $\Vv(n)$ of words in generators of length at 
most $n$ has a basis 
$$
\{x^k y^l, \, x^i (y x) y^j, \; |\; k+l \leq n, \; i+j \leq n-2\}.
$$ 
Consequently, 
$$
\dim \Vv(n) = {{n+2}\choose{2}} + {{n}\choose{2}} = n^2 + n +1,
$$
so $\opil$ has the Gelfand-Kirillov dimension two. 

We now turn to the description of $\cO(\TT^2_\theta)_-$. Let 
\begin{equation}\label{hats}
\hx = V-V^*, \qquad  \hy=W-W^*, \qquad \mbox{and} \quad  \hz=VW^* - V^*W.
\end{equation}
Clearly, $\hx, \hy, \hz \in \cO(\TT^2_\theta)_-$, and, in fact 
\begin{lemma}\label{lem.a-}
The elements $\hx, \hy, \hz$ generate $\cO(\TT^2_\theta)_-$ as a left (resp.\hspace{3pt}right) $\opil$-module.
\end{lemma}
\begin{proof}
$\cO(\TT^2_\theta)_-$ is spanned by 
\begin{equation}\label{hab}
\ha_{m,n} = V^mW^n -V^{*m}W^{*n},\quad \hb_{m,n} = V^mW^{*n} -V^{*m}W^{n}, \quad m,n\in \NN,\; m+n>0.
\end{equation}
We prove by induction that $\ha_{m,n}$ are elements of the left $\opil$-module generated by $\hx, \hy, \hz$ (the proof for $\hb_{m,n}$ is similar). First,
$\ha_{1,0} = \hx$, $\ha_{0,1} = \hy$ and
$$
\ha_{1,1} = VW - V^*W^*  = (V+V^*)(W-W^*) - V^*W +VW^* = x\hy + \hz.
$$
Next, we fix an $n$ and assume that
\begin{equation}\label{ind}
\ha_{m,n} = a^x_{m,n}\hx +  a^y_{m,n}\hy + a^z_{m,n}\hz, \qquad a^x_{m,n},  a^y_{m,n}, a^z_{m,n} \in \opil,
\end{equation}
for all $m\leq k$. Then, in view of the unitarity of $V$,
\begin{eqnarray*}
\ha_{k+1, n} &=& V^{k+1}W^n -V^{*k+1} W^{*n} \\
&=& (V+V^*) (V^{k}W^n - V^{*k} W^{*n}) - V^{k-1}W^n +  V^{*k-1} W^{*n}\\
&=& (xa^x_{k,n} - a^x_{k-1,n})\hx +  (xa^y_{k,n} - a^y_{k-1,n})\hy + (xa^z_{k,n} - a^z_{k-1,n})\hz,
\end{eqnarray*}
hence $\ha_{k+1,n}$ is in the module generated by  $\hx, \hy, \hz$. Using the unitarity of $V$ and $W$ and relations \eqref{torus} one finds:
\begin{equation}\label{hat.y}
\hx y = \lambda y\hx + (1-\lambda^2)\hz , \qquad \hy y = y\hy, \qquad  \hz y =\bar{ \lambda} y\hz - (1-\bar{\lambda}^2)\hx .
\end{equation}
Let us fix an $m$ and assume that \eqref{ind} is true for all $n\leq k$. Then, by the unitarity of $W$,
\begin{eqnarray*}
\ha_{m, k+1} &=& V^{m}W^{k+1} -V^{*m} W^{*k+1} \\
&=& (V^{m}W^{k} -V^{*m} W^{*k})(W+W^*) - V^{m}W^{k-1} +V^{*m} W^{*k-1}\\
&=& a^x_{m,k}\hx y +  a^y_{m,k}\hy y + a^z_{m,k}\hz y - (a^x_{m,k-1}\hx +  a^y_{m,k-1}\hy + a^z_{m,k-1}\hz).
\end{eqnarray*}
In view of the relations \eqref{hat.y}, $\ha_{m,k+1}$ is in the left $\opil$-module generated by $\hx, \hy, \hz$.
\end{proof}

We are now ready to construct a connected two-dimensional differential calculus $\Omega(\pil)$ on $\opil$. Set $\Omega^1 (\pil) = \cO(\TT^2_\theta)_-\oplus \cO(\TT^2_\theta)_-$ and $\Omega^ 2 (\pil) = \opil$, $\Omega^n (\pil) =0$, for all $n>2$, and define the product of elements in $\Omega^1 (\pil)$, by
\begin{equation}\label{prod.b}
(a_1, a_2)\wedge (a_3, a_4) = a_1a_4 - a_2a_3\in \opil, \qquad a_k\in \cO(\TT^2_\theta)_-
\end{equation}
The following linear endomorphisms of $\cO(\TT^2_\theta)$, 
\begin{equation}\label{partial}
\partial_V (V^mW^n) = imV^mW^n, \qquad \partial_W (V^mW^n) = inV^mW^n,
\end{equation} 
are derivations i.e.\hspace{3pt}they satisfy the Leibniz rule. The factor $i$ in the above formulae ensures that $\partial_{V,W}\circ * = *\circ \partial_{V,W}$. Furthermore $\partial_{V,W}$ commute among themselves and anticommute with the automorphism $\sigma$ \eqref{sigma.auto}, i.e.\hspace{3pt}
$$
 \partial_V\circ \partial_W = \partial_W\circ \partial_V \quad \mbox{and} \quad \partial_{V,W} \circ \sigma = - \sigma \circ \partial_{V,W}.
$$
All this means that $\partial_{V,W} (\cO(\TT^2_\theta)_\pm) \subseteq \cO(\TT^2_\theta)_\mp$ and that the maps
$$
d: \opil \to \Omega^1 (\pil), \quad a\mapsto (\partial_V(a), \partial_W(a)), 
$$
and
$$
d:  \Omega^1 (\pil)\to \opil, \quad (a_1,a_2)\mapsto \partial_V(a_2) - \partial_W(a_1),
$$
define a complex. The combination of the derivation property of $\partial_{V,W}$ together with the definition of multiplication  in \eqref{prod.b} ensure that the maps $d$ satisfy the graded Leibniz rule. Clearly $\partial_V (V^mW^n) =\partial_W (V^mW^n) = 0$ if and only if $m=n=0$, hence $\Omega(\pil)$ is a connected calculus. With these at hand we can state:

\begin{theorem}\label{thm.pillow.Poincare}
$\Omega(\pil)$ is a connected integrable differential calculus, hence the  noncommutative pillow algebra is differentially smooth.
\end{theorem}
\begin{proof}
In order to apply Lemma~\ref{lemma.Poincare} we need to check if $\cO(\TT^2_\theta)_-\cO(\TT^2_\theta)_- = \opil$ and that $\Omega (\pil)$ is a differential calculus (not just a differential graded algebra) over $\opil$. To check the former, observe that
\begin{equation}\label{xyz}
\hat{x}^2 + \hat{y}^2 -\bar{\lambda}\hat{z}^2 - \hat{x}z\hat{y} = 2(\bar{\lambda}^2-1),
\end{equation}
which ensures that $1\in \cO(\TT^2_\theta)_-\cO(\TT^2_\theta)_-$ and, consequently, that $\cO(\TT^2_\theta)_-\cO(\TT^2_\theta)_- = \opil$. Straightforward calculations that use the defining relations of the noncommutative torus lead to the following relations supplementing \eqref{hat.y}
\begin{subequations}\label{hat.xz}
\begin{equation}\label{hat.x}
\hx x = x\hx, \qquad \hy x = \bar{\lambda} x\hy - (\lambda-\bar{\lambda})\hz , \qquad  \hz x ={ \lambda} x\hz + (\lambda-\bar{\lambda})\hy .
\end{equation}
\begin{equation}\label{hat.z}
\hx z = \bar{\lambda} z\hx + (1-\bar{\lambda}^2)\hy ,  \qquad  \hy z ={ \lambda} z\hy - (\lambda-\bar{\lambda})\hx , \qquad \hz z = z\hz .
\end{equation}
\end{subequations}
In view of the definitions  of $\partial_{V,W}$  and $d$ one finds $(i\hx , 0) = dx$ and $(0,i\hy) = dy$. This observation combined with equations \eqref{hat.y} and \eqref{hat.xz} yields
$$
 (i\hy, 0) = \frac{1}{1-\bar{\lambda}^2} (dx z - \bar{\lambda}z dx), \qquad 
(i\hz, 0) =  \frac{1}{1-{\lambda}^2} (dx y - {\lambda}y dx), 
$$
and
$$
(0,i\hx) = \frac{1}{\bar{\lambda}-{\lambda}} (dy z - {\lambda}z dy), \qquad (0,i\hz) = \frac{1}{1-{\lambda}^2} (\lambda dy x - x dy).
$$
Since $\cO(\TT^2_\theta)_-$ is generated by $\hx, \hy, \hz$ and $d$ satisfies the Leibniz rule, this proves that $\Omega^1 (\pil) = \opil d\opil$. Finally,
\begin{eqnarray*}
&&\frac{1}{2(\blambda^2 -1)}\left( (\hx, 0)(0,\hx) + (\hy, 0)(0, \hy) - \lambda(\hz, 0)(0,\hz) - (\hx, 0)(0,z\hy)\right)\\
&& = \frac{1}{2(\blambda^2 -1)}\left(\hat{x}^2 + \hat{y}^2 -\bar{\lambda}\hat{z}^2 - \hat{x}z\hat{y}\right) =1,
\end{eqnarray*}
by \eqref{xyz}. Hence
$$
1\in \Omega^1 (\pil)\wedge \Omega^1 (\pil) = \opil d\opil\wedge  \opil d\opil = \opil d\opil \wedge d\opil,
$$
which implies that $ \opil d \opil \wedge d \opil=\opil=\Omega^2(\pil)$, as required for a differential calculus. Therefore, Lemma~\ref{lemma.Poincare} implies that $\Omega(\pil)$ is an integrable differential calculus (with integrating form 1) over the noncommutative pillow as required.
\end{proof}

Since 
$$
\partial_V(\ha_{m,n}) = ima_{m,n}, \quad \partial_W(\ha_{m,n}) = ina_{m,n}, \quad \partial_V(\hb_{m,n}) = imb_{m,n}, \quad \partial_W(\hb_{m,n}) = -inb_{m,n}, 
$$
where $a_{m,n}$, $b_{m,n}$ are given by \eqref{ab} and $\ha_{m,n}$, $\hb_{m,n}$ are given by \eqref{hab}, the image of the divergence $\nabla$ which obviously coincides with the image of $d: \Omega^1(\pil)\to \opil$, contains all elements of $\opil$ except those that are in the subspace spanned by the identity, i.e. $\opil = \CC \oplus \im\nabla$. Therefore $\coker \nabla =\CC$ and the integral $\Lambda: \opil\to \CC$ comes out as
$$
\Lambda(a) = 
\begin{cases}
\alpha, & \mbox{if $a=\alpha 1$, for  $\alpha \in \CC$} \\
0, & \mbox{otherwise},
\end{cases}
$$
i.e.\hspace{3pt}it is equal to the trace on $\opil$ obtained by the restriction of the trace on the noncommutative torus.

\begin{remark}
Theorem~\ref{thm.pillow.Poincare} can be also proven with the help of Lemma~\ref{lem.princ}. The $\ZZ_2$-action on the noncommutative 
two-torus  can be interpreted as  the following coaction of $\CC \ZZ_2$ on  $\cO(\TT^2_\theta)$, 
$$ 
\varrho(V) = \frac{1}{2} (V\! +\! V^*) \otimes 1 + \frac{1}{2} (V\! -\! V^*) \otimes u, \;\;
\varrho(W) = \frac{1}{2} (W\! +\! W^*) \otimes 1 + \frac{1}{2} (W\! -\! W^*) \otimes u,
$$
where $u$ is the generator of $\ZZ_2$, $u^2=1$. The coinvariant subalgebra coincides with $\cO(\pil)$. Equation \eqref{xyz} assures that this coaction is principal.   Furthermore, it extends to the standard calculus $\Omega(\TT^2_\theta)$ on $\cO(\TT^2_\theta)$ freely generated by two central, anticommuting forms $\omega_V = V^*dV$ and $\omega_W = W^*dW$. One can check that $\Omega(\TT^2_\theta)$ is integrable with a volume form $\omega_V\wedge \omega_W$, invariant under the coaction $\varrho$. Furthermore, one can verify that the invariant part of $\Omega(\TT^2_\theta)$ coincides with the calculus $\Omega(\pil)$ constructed above and that the condition (c) in  Lemma~\ref{lem.princ} is also satisfied.
\end{remark}

\subsection{The noncommutative pillow as a complex manifold}
As before we use the direct sum decomposition $\cO(\TT^2_\theta) = \cO(\TT^2_\theta)_+\oplus \cO(\TT^2_\theta)_-$, with  $\opil=\cO(\TT^2_\theta)_+$, determined by the automorphism \eqref{sigma.auto}.  $(\Omega (\pil), d)$ is the differential calculus constructed in Section~\ref{sec.dif}.

\begin{proposition}\label{prop.cotan}
The module $\cO(\TT^2_\theta)_-$ is a non-free finitely generated projective left (resp.\hspace{3pt}right) module over the non-commutative pillow algebra $\opil$. Consequently, the cotangent bundle over $\pil$ whose sections are given by $\Omega^1(\pil)$ is non-trivial.
\end{proposition}
\begin{proof}
The fact that $\cO(\TT^2_\theta)_-\cO(\TT^2_\theta)_-=\opil$ means  that $\cO(\TT^2_\theta)_-$ is a finitely 
generated projective left and right $\opil$-module. The formula \eqref{xyz} yields an idempotent
\begin{equation}\label{idempotent}
\mathbf{e} = \frac{1}{2(\bar{\lambda}^2 -1)} 
\begin{pmatrix} \hx \cr \hy \cr \hz \end{pmatrix}
\begin{pmatrix} \hx, & \hy -\hx z, & -\bar\lambda \hz \end{pmatrix},
\end{equation}
which identifies $\cO(\TT^2_\theta)_-$ as a submodule of the free module $\opil^3$. The matrix trace of $\mathbf{e}$ comes out as
$$
\tr \mathbf{e} = 3 - \frac12 (x^2+\lambda z^2) = 1 - \frac12 (a_{2,0}+\lambda^2b_{2,2}),
$$
where $a_{m,n}$, $b_{m,n}$ are given by \eqref{ab}. Therefore, $\Lambda(\tr \mathbf{e}) = 1$. However, since $\Lambda$ is a trivial trace this does not yet guarantee that $\cO(\TT^2_\theta)_-$ is not a free $\cO(\pil)$-module. To show the  non-freeness we need to consider a family of traces on $\opil$, which arise from twisted traces on the algebra of the noncommutative torus. More specifically, if for a given involutive automorphism $\sigma$ of $\cO(\TT^2_\theta)$ there is a functional $\htau$ such that
$$ \htau ( a b) = \htau( \sigma(b) a), \;\;\; \forall a,b \in  \cO(\TT^2_\theta), $$
then $\htau$ restricted to the invariant subalgebra is a trace. For the automorphism \eqref{sigma.auto}, there are four linearly independent twisted traces (in addition to $\Lambda$) \cite{Walters}: 
$$ \htau_{ij} (V^\alpha W^\beta) = e^{\pi i \theta \alpha \beta} \delta_i^{\bar\alpha} \delta_j^{\bar\beta}, \qquad  i,j =0,1,
$$
where $\bar \alpha = \alpha \! \mod \! 2$ and $\bar \beta = \beta \! \mod \! 2$. Evaluating the extension of these traces to the  projector \eqref {idempotent} we obtain the following
results:
$$ \htau_{00}( \tr \mathbf{e}) =  - 1, 
\;\;\; \htau_{01}( \tr \mathbf{e}) = 0, 
\;\;\; \htau_{10}( \tr \mathbf{e}) = 0, \;\;\; 
\htau_{11}(\tr \mathbf{e}) = 0. $$
It is the nontrivial value of $\htau_{00}$ on $ \tr \mathbf{e}$, which shows that the module $\cO(\TT^2_\theta)_-$ determined by $\mathbf{e}$ is non-free. Therefore, the module of sections of the cotangent bundle $\Omega^1(\pil)$, which is given as 
$\cO(\TT^2_\theta)_- \oplus \cO(\TT^2_\theta)_-$, is also non-trivial.
\end{proof}

Following \cite{KhaLan:hol} (cf.\ \cite{BegSmi:com}), by a complex structure on $\opil$ corresponding to the differential calculus $(\Omega (\pil), d)$  we understand the bi-grading decomposition of $\Omega (\pil)$,
$$
\Omega^n(\pil) = \bigoplus_{p+q =n} \Omega^{(p,q)}(\pil),
$$
a $*$-algebra structure on $\Omega (\pil)$ such that $*:\Omega^{(p,q)}(\pil)\to \Omega^{(q,p)}(\pil)$, and the decomposition $d= \delta +\bdelta$ into differentials $\delta: \Omega^{(p,q)}(\pil)\to \Omega^{(p+1,q)}(\pil)$, $\bdelta: \Omega^{(p,q)}(\pil)\to \Omega^{(p,q+1)}(\pil)$ such that 
\begin{equation}\label{partial.bpartial}
\delta(a)^* = \bdelta(a^*), \qquad \mbox{for all $a\in \Omega (\pil)$}.
\end{equation}
As explained in \cite{KhaLan:hol}, up to a conformal factor in a metric, the complex structure associated to the calculus on $\cO(\TT^2_\theta)$ corresponding to derivations $\partial_V$ and $\partial_W$ is determined by the derivations
\begin{equation}\label{partial.tau}
\partial_\tau= \frac{1}{\tau-\btau}(-\btau \partial_V +\partial_W), \qquad \bpartial_\tau= \frac{1}{\tau-\btau}(\tau \partial_V -\partial_W),
\end{equation}
where $\tau \in \CC\setminus\RR$. With the help of these derivations one constructs complex structures for the calculus $\Omega (\pil)$ over the non-commutative pillow manifold as follows: 
$$\Omega_\tau^{(1,0)}(\pil) = \Omega_\tau^{(0,1)}(\pil) =\cO(\TT^2_\theta)_-, \quad \Omega_\tau^{(2,0)}(\pil) = \Omega_\tau^{(0,2)}(\pil) =0, \quad \Omega_\tau^{(1,1)}(\pil) = \opil.
$$
 The $*$-structure on $\cO(\TT^2_\theta)_-$ is that of the $*$-structure of the noncommutative torus, while the  $*$-structure on $\Omega_\tau^{(1,1)}(\pil) $ is opposite to that on $\opil$. The holomorphic and anti-holomorphic differentials are
 $$
 \delta_\tau:= \partial_\tau\mid_\opil: \opil\to \Omega_\tau^{(1,0)}(\pil), \qquad \delta_\tau:= \partial_\tau\mid_{\cO(\TT^2_\theta)_-}:\Omega_\tau^{(0,1)}(\pil)\to \opil,
 $$
 and
 $$
 \bdelta_\tau:= \bpartial_\tau\mid_\opil: \opil\to \Omega_\tau^{(0,1)}(\pil), \qquad \bdelta_\tau:= -\bpartial_\tau\mid_{\cO(\TT^2_\theta)_-}:\Omega_\tau^{(1,0)}(\pil)\to \opil.
 $$
 One can understand this complex structure as the decomposition of the differential $d$ discussed in Section~\ref{sec.dif} as follows. Let us write $(\ha,\hb)_\tau$ for an element of $\Omega_\tau^{(1,0)}(\pil)\oplus \Omega_\tau^{(0,1)}(\pil)$ and $\omega_\tau$, $\bomega_\tau$ for a basis of the direct sum  $\Omega_\tau^{(1,0)}(\pil)\oplus \Omega_\tau^{(0,1)}(\pil)$, so that
 $$
 (\ha,\hb)_\tau = \ha\omega_\tau +\hb\bomega_\tau, \qquad \ha,\hb\in \cO(\TT^2_\theta)_-.
 $$
 Similarly, let us write $(a,b)$ for an element of $\Omega^1 (\pil)$ and $\omega_V$, $\omega_W$ for a basis of the direct sum decomposition into the $\cO(\TT^2_\theta)_-$. In this notation, for all $a\in \opil$,
 $$
  \partial_\tau(a)\omega_\tau + \bpartial_\tau(a)\bomega_\tau = da =   \partial_V(a)\omega_V + \partial_W(a)\omega_W.
  $$
  By comparing coefficients at $\partial_V$ and $\partial_W$ one arrives at the following invertible transformation
  $$
  \begin{pmatrix}
  \omega_\tau \cr \bomega_\tau 
  \end{pmatrix}
  = \frac{1} {\tau-\btau} \begin{pmatrix} 1 & \tau \cr 1 & \btau \end{pmatrix} \begin{pmatrix}
  \omega_V \cr \omega_W
  \end{pmatrix},
$$
i.e.\
$$
(\ha,\hb)_\tau = \frac{1} {\tau-\btau} (\ha +\hb, \tau\ha +\btau\hb).
$$
In particular, this interpretation guarantees that the differential graded algebra determined from $\Omega^{(p,q)}_\tau (\pil)$ 
is a differential calculus over $\opil$.

\subsection{Spectral geometry of the noncommutative pillow}

In the classical situation the quotient of the torus by the action of the cyclic group $\ZZ_2$, which was described 
in Section~\ref{sec.dif}, gives an orbifold with four corners. To see the striking difference between the commutative 
and noncommutative cases it is convenient to use the same language, adapted to describe both commutative and 
noncommutative manifolds. At present, the best candidate for such an approach is the notion of a spectral triple 
\cite{Connes}, which is modelled on the definition of Riemannian spin geometry.

In what follows, we shall briefly recall the construction of a spectral triple for the noncommutative torus (cf.\ \cite{Dab:spi}), 
then we shall find real spectral triples for the noncommutative pillow and see whether the axioms of spectral triples 
allow us to determine whether the latter is a manifold or an orbifold.

Let $\storus$ denote the Fr\'echet algebra of smooth elements of the noncommutative torus, which contains
all series $ \sum_{m,n} a_{mn} V^m W^n $ with $\{a_{mn}\}_{m,n \in \ZZ} \in \CC$ a rapidly decreasing sequence. 
Fix $\epsilon = 0$ or $\epsilon = \frac{1}{2}$. Consider a separable Hilbert space $\cH$ with an orthonormal basis 
$e_{m,n}$,  $m,n \in \ZZ \!+\! \epsilon$ (so there are four possibilities) and the representation $\pi$ of the 
algebra $\storus$ given on the generators of the noncommutative torus as:
$$ \pi(V) e_{m,n} = e^{\pi i \theta n} e_{m+1,n}, \;\;\;\;\;\;\; \pi(W) e_{m,n} = e^{-\pi i \theta m} e_{m,n+1}. $$

\begin{theorem}[\cite{PaSi}]
The following datum $(\storus, \cH \otimes \CC^2, D, \gamma, J)$ gives a real equivariant spectral
triple over the noncommutative torus. Here, the representation of the algebra is taken to be diagonal,
the operators $\gamma$ and $J$ are:
$$ \gamma e_{m,n,\pm} = \pm e_{m,n,\pm},  \;\;\;\;\;\;\; J e_{m,n,\pm} = \mp e_{-m,-n,\mp}, $$
and the Dirac operator $D$ is:
$$ D  e_{m,n,-} = (m + \tau n)  e_{m,n,+}, \;\;\;\;\; D  e_{m,n,+} = (m + \tau^* n)  e_{m,n,-}, $$
for any $\tau$, $|\tau|=1$, with a non-zero imaginary part.
\end{theorem}

Observe that four different possibilities for the choice of $m,n$ (integer or half-integer) correspond to four different 
spin structures over the classical torus \cite{PaSi}. Next, following the ideas of \cite{Adalgott} we can construct the restriction 
of these spectral triples to the invariant subalgebra of $\storus$ with respect to  the $\ZZ_2$-action determined by 
(the extension of)  the automorphism $\sigma$ \eqref{sigma.auto}. As the first step we lift the action to the Hilbert space 
$\cH \otimes \CC^2$.

\begin{lemma}
For any of the four spin structures and any $\tau$, 
$$ \rho e_{m,n,\pm} = \pm e_{-m,-n,\pm}, $$
is unique (up to sign) lift of the action of $\ZZ_2$  to the Hilbert space  $\cH \otimes \CC^2$, which commutes with $J$, $D$ and $\gamma$ and implements the action 
of $\sigma$ \eqref{sigma.auto} on the algebra $\storus$.
\end{lemma}

This lemma is a direct consequence of the application of Theorem 2.7 of \cite{Adalgott} when one restricts to a subalgebra
of the noncommutative torus.  As a consequence, we have:

\begin{corollary}
Each of the spin structures over the noncommutative torus restricts to an irreducible real spectral triple over the algebra of the 
noncommutative pillow by taking $\spil = \storus_+, (\cH \otimes \CC^2)_+, D, \gamma, J$, where
$\storus_+$ is the invariant subalgebra of (the extension of)  the automorphism $\sigma$ \eqref{sigma.auto},  $(\cH \otimes \CC^2)_+$ is the invariant part of the Hilbert space (the eigenspace of
$\rho$ with eigenvalue $1$) and $D,J,\gamma$ are restrictions of the operators to that space.
\end{corollary}

So far we have constructed and classified all real spectral triples over the noncommutative pillow, which are restrictions of equivariant
real spectral triples over the noncommutative torus. Such a construction guarantees that almost all conditions, which are satisfied for a spectral
triple over the torus (such as finiteness and the behaviour of the spectrum of the Dirac operator, etc.) are automatically satisfied. There exists
however one exception, which is the crucial element making it possible to distinguish manifolds from orbifolds. This is the orientability condition
(see \cite{ReVa}), which is satisfied for all commutative spectral geometries and for the noncommutative torus. The original condition requires the existence of a Hochschild cycle, so that its image under $\pi$ gives the chirality operator $\gamma$ (in the even-dimensional case). For the noncommutative pillow we can establish a milder version of the orientability condition, by constructing explicitly a cycle
(though not a Hochschild cycle) whose image is the chirality operator $\gamma$.
\begin{theorem}
For any value of the conformal structure $\tau$ with the non-zero imaginary part, $|\tau|=1$, and any value of the deformation parameter
$\lambda = e^{2 \pi i \theta}$,  provided that $\lambda^4 \not= 1$, there exists a cycle 
$\omega = \sum_i a_i \ot b_i \ot c_i \in \spil \ot \spil \ot \spil$, 
such that:
 $$\sum \pi(a_i)[D,\pi(b_i)][D,\pi(c_i)] = \gamma .$$ 
\end{theorem}
\begin{proof}
We write $\tau = \cos \phi + i \sin \phi$. Then by explicit computation we verify that
$$
\begin{aligned}
\frac{1}{4}&(1 - \lambda^4) x [D,y] [D,z] \\
& - \frac{1}{2}  \lambda^2 \cos\phi (1+\lambda^2)(2  \lambda^4 \cos\phi-\lambda^4 +2\cos\phi -1)\, x [D,z] [D,y] 
 -2\lambda^2 \\
& - \frac{1}{2} (2 \lambda^6 \cos^2\phi   -\lambda^6 \cos\phi +2 \lambda^4\cos^2\phi  + 2\lambda^4 \\
& \quad \quad   - 5 \lambda^4\cos\phi   
-  \lambda^3 \cos\phi + 2  \lambda^2 \cos^2\phi + 2 \cos^2\phi - \cos\phi) \lambda^2 \, y [D,x] [D,z] \\
&+ \frac{1}{2} \lambda (1+\lambda^2)(1+\lambda^4)(2\cos\phi -1)(\cos\phi -1)\, z [D,x] [D,y] \\
& - \frac{1}{2} (1-\lambda^4) (\cos\phi -1) \lambda^2 \, z [D,y] [D,x] \\
&+\frac{1}{4} (1+\lambda^2)(2\cos\phi -1)(2\lambda^4\cos\phi  +2 \cos\phi - \lambda^4 -\lambda^2) \lambda \, [D,x] y [D,z] \\
&+\frac{1}{2} \cos\phi \lambda^2 (2 \lambda^6 \cos\phi + 2 \lambda^4\cos\phi  + 2 \lambda^2 \cos\phi  + 2 \cos\phi \\
&\quad\quad - \lambda^6 - \lambda^2  -2 -4 \lambda^4) \, [D,z] [D,y] x \\
&- \frac{1}{2} \lambda^2 (1+\lambda^2)(1+\lambda^4)(2 \cos\phi -1)(\cos\phi -1) \, [D,x] z [D,y] =\\
& = \sin\phi (2\cos\phi -1) (\lambda^2 -1)^3 (1+\lambda^2) \gamma.
\end{aligned}
$$
Although the above formula uses terms like $[D,x] y [D,z]$, it can be easily converted to the desired form using
the Leibniz rule.
\end{proof}

\begin{remark}
The above cycle is certainly not unique and we provide its explicit form only to demonstrate its existence. Certainly, it 
could be simplified or rewritten in a more convenient form. 

Observe that for the image of this cycle to be nonvanishing we must have $\sin\phi \not= 0$ (which is the same condition 
as for the noncommutative torus, see \cite{PaSi}) as well as $\lambda^4 \not= 1$. It is interesting that even for some 
rational values of $\theta$ we have a non-degenerate cycle for the respective spectral triple.
\end{remark}

\begin{remark}
The above demonstrated cycle is not a Hochschild cycle. It remains an open problem whether a Hochschild cycle of the same type
exists for the noncommutative pillow. However, we would like to stress that in \cite{ReVa} the nonexistence of an orientation cycle for an 
orbifold is demonstrated for any cycle (not only for the Hochschild ones).
\end{remark}

\section{Quantum cones}\setcounter{equation}{0}
\subsection{Two-dimensional integrable differential calculi over quantum cones}\label{sec.dif.cone}
Homological smoothness and the complex differential geometry of quantum cones were recently established and studied in 
\cite{Brz:com}. In this section we prove that the  two-dimensional differential calculus over the quantum cone described in \cite{Brz:com} is integrable and thus establish differential smoothness of quantum cone algebras. 

Let $N$ be a positive integer. The coordinate algebra of the {\em quantum cone} $\cO(C^N_{q,\kappa})$ is defined as a complex $*$-algebra generated by self-dual $a$ and $b, b^*$, which satisfy relations
\begin{equation}\label{cone}
ab = q^N ba + \kappa [N]_q b, \qquad bb^* = \prod_{l=0}^{N-1} (q^{-l}a + \kappa [-l]_q), \qquad  b^*b = \prod_{l=1}^{N} (q^{l}a + \kappa [l]_q),
\end{equation}
where $q>0$, $\kappa \in \RR$ are parameters and, for all $n\in \ZZ$,  
\begin{equation}\label{q.integers}
[n]_q := \frac{1-q^{n}}{1-q}
\end{equation}
denotes $q$-integers. For $N\neq 1$, a linear basis of the space $\Vv(n)$ spanned by words of generators of length at most $n$ is
$$
\{a^ib^j, \, a^{k}b^{*l+1} \; |\; i+j \leq n ,\, k+ l \leq n-1, \},
$$
hence
$$
\dim \Vv(n) = {n+2 \choose n} +  {n+1 \choose n-1} = (n+1)^2.
$$
Consequently, 
$$
\gk (\cO(C^N_{q,\kappa})) =2.
$$
For $N=1$, $a$ can be expressed in terms of $bb^*$, so it becomes redundant. In this case we denote $b=z$ and $b^*=z^*$. Then relations become
\begin{equation}\label{disc}
z^*z - qzz^* = \kappa,
\end{equation}
 and the resulting algebra is the coordinate algebra of the {\em quantum disc} \cite{KliLes:two}, which we denote by $\cO(D_{q,\kappa})$. The linear basis is $z^iz^{*j}$, $i,j\in \NN$, hence the quantum disc algebra has  Gelfand-Kirillov dimension two. The special case $\kappa =0$ corresponds to the Manin quantum plane and the case $q=1$ and $\kappa \not= 0$ corresponds to the Moyal
 deformation of the plane.
 
 A differential $*$-calculus $\Omega (D_{q,\kappa})$ on $\cO(D_{q,\kappa})$ is freely generated by one-forms $dz$ and $dz^*$ subject to relations
\begin{equation}\label{calculus.d}
zdz = q^{-1}dz z, \qquad z^*dz = qdz z^*, \qquad dz\wedge dz^* = -q^{-1}dz^*\wedge dz, \qquad dz\wedge dz =0,
\end{equation}
and their $*$-conjugates; see e.g.\hspace{3pt}\cite{SinVak:ana}. Since $\Omega^1(D_{q,\kappa})$ is an $\cO(D_{q,\kappa})$-bimodule freely generated by  $dz$ and $dz^*$, the commutation relations between $dz$ and elements $\alpha$ of $\cO(D_{q,\kappa})$ induce an algebra automorphism $\sigma : \cO(D_{q,\kappa}) \to \cO(D_{q,\kappa})$, by 
\begin{equation}\label{auto}
\alpha dz = dz \sigma(\alpha), \qquad \mbox{for all $\alpha\in \cO(D_{q,\kappa})$}. 
\end{equation}
On the basis $z^kz^{*l}$, $k,l\in \NN$, of $\cO(D_{q,\kappa})$ this comes out as
\begin{equation}\label{auto.explicit}
\sigma(z^kz^{*l}) = q^{l-k} z^kz^{*l}.
\end{equation}
Note that   $\sigma$ also satisfies the equality 
\begin{equation}\label{auto*}
\alpha dz^* = dz^* \sigma(\alpha), \qquad \mbox{for all $\alpha \in \cO(D_{q,\kappa})$}.
\end{equation}
Finally, the commutation rules \eqref{auto} and \eqref{auto*} imply that $d(\alpha) =0$ if and only if $\alpha$ is a scalar multiple of the identity, i.e.\hspace{3pt}the calculus $\Omega (D_{q,\kappa})$ is connected.

 For all values of $N$, $\cO(C^N_{q,\kappa})$ embeds into $\cO(D_{q,\kappa})$ by the $*$-inclusion
\begin{equation}\label{inclusion}
a\mapsto zz^*, \qquad b\mapsto z^N.
\end{equation}
We define the calculus $\Omega (C^N_{q,\kappa})$ on $\cO(C^N_{q,\kappa})$, by restricting $\Omega (D_{q,\kappa})$ to $\cO(C^N_{q,\kappa})$. 

 \begin{theorem}\label{thm.cone.smooth}
 For all $N\in \NN$ and $\kappa\neq 0$, $\Omega (C^N_{q,\kappa})$ is a 2-dimensional connected integrable differential calculus on $\cO(C^N_{q,\kappa})$, hence the quantum cone algebras $\cO(C^N_{q,\kappa})$ are differentially smooth.
 \end{theorem}
 \begin{proof}
As a restriction of a connected calculus $\Omega (C^N_{q,\kappa})$ is connected.  $\cO(D_{q,\kappa})$ can be equipped with a grading by  elements of the cyclic group $\ZZ_N = \{0,1,\ldots, N-1\}$, defined by $\deg(z) =1$, $\deg(z^*) = N-1$, which is compatible with the $*$-operation in the sense that if $\deg(\alpha) = k$, then $\deg(\alpha^*) = N-k$. The embedding \eqref{inclusion} identifies $\cO(C^N_{q,\kappa})$ with the invariant subalgebra $\cO(D_{q,\kappa})_0$ of $\cO(D_{q,\kappa})$. By \cite[Theorem~2.1]{Brz:com} this grading is strong, provided $\kappa \neq 0$.
Note that $\sigma$ \eqref{auto} preserves the $\ZZ_N$-grading of $\cO(D_{q,\kappa})$. 

By \cite[Theorem~3.1]{Brz:com}, the restriction $\Omega (C^N_{q,\kappa})$ of $\Omega (D_{q,\kappa})$ to the differential calculus  on $\cO(C^N_{q,\kappa})$ is generated by the one-forms 
$$
dz z^*, \qquad db \propto dz z^{N-1}, \qquad dz^* z, \qquad db^* \propto dz^* z^{*N-1}.
$$
The first two one-forms generate the holomorphic part of $\Omega^1 (C^N_{q,\kappa})$, while the other two generate the anti-holomorphic part. The module of one-forms $\Omega^1 (C^N_{q,\kappa})$ is not free, but $\Omega^2 (C^N_{q,\kappa})$ is freely generated by the closed volume form $dz\wedge dz^*$:
$$ da \wedge da = (z \, dz^* + z^* \, dz)  \wedge ( z \, dz^* + z^* \, dz)  = (- z z^* + \frac{1}{q}  z^* z) \, dz \wedge dz^*  = 
- \frac{\kappa}{q} \, dz \wedge dz^*, $$ 
hence it can be identified with  $\cO(C^N_{q,\kappa})$.

Since $z^*$ and $z^{N-1}$ generate the $\cO(C^N_{q,\kappa})$-submodule of $\cO(D_{q,\kappa})$ consisting of elements of $\ZZ_N$-degree $N-1$, the holomorphic part of $\Omega^1 (C^N_{q,\kappa})$ consists of elements of the form
$$
dz \alpha_-, \qquad \deg(\alpha_-) = N-1.
$$
Similarly, the antiholomorphic forms are $dz^* \alpha_+$, $\deg(\alpha_+)=1$. This shows that 
$$
\Omega^1 (C^N_{q,\kappa}) \cong \cO(D_{q,\kappa})_1 \oplus \cO(D_{q,\kappa})_{N-1},
$$
where the right module isomorphism is
\begin{equation}\label{iso.omega}
dz\alpha_- +dz^*\alpha_+ \mapsto (\alpha_+, \alpha_-), \qquad \deg(\alpha_\pm) = \pm 1\, \mbox{mod}\, N.
\end{equation}
By relations  \eqref{calculus.d}, \eqref{auto} and \eqref{auto*}, 
\begin{equation}\label{product.cone}
(dz\alpha_- +dz^*\alpha_+)(dz\beta_- +dz^*\beta_+) = dz\wedge dz^*\left(\sigma(\alpha_-)\beta_+ -q\sigma(\alpha_+)\beta_- \right).
\end{equation}
In view of the isomorphism \eqref{iso.omega}, the product rule \eqref{product.cone} can be recast into the desired form
$$
(\alpha_+,\alpha_-)\wedge (\beta_+,\beta_-) = \sigma_+(\alpha_+)\beta_- + \sigma_-(\alpha_-)\beta_+,
$$
where 
$$
\sigma_+ = -q \sigma, \qquad \sigma_- = \sigma,
$$
with $\sigma$ given by \eqref{auto}. In this way the constructed calculus over the quantum cone meets all of the requirements of Lemma~\ref{lemma.Poincare}, and therefore it is integrable. 
\end{proof}

\begin{remark}
The integrability of the calculi over the quantum cones can also be established by employing Lemma~\ref{lem.princ}. The $\ZZ_N$-grading of $\cO(D_{q,\kappa})$ which yields $\cO(C^N_{q,\kappa})$ can be interpreted as the group algebra $\CC \ZZ_N$-coaction
$$
\varrho(z) = z\ot u, \qquad \varrho(z^*) = z^*\ot u^{N-1},
$$
where $u^N= 1$ is the generator of $\ZZ_N$ written multiplicatively. This coaction, which is principal by \cite[Theorem~2.1]{Brz:com}, extends to the differential calculus $\Omega(D_{q,\kappa})$, whose invariant part coincides with $\Omega(C^N_{q,\kappa})$ and satisfies assumptions (a)--(c) of Lemma~\ref{lem.princ}. Once it is shown that $\Omega(D_{q,\kappa})$ is an integrable calculus, the integrability of $\Omega(C^N_{q,\kappa})$ will follow from Lemma~\ref{lem.princ}.
\end{remark}

Using  \eqref{calculus.d} one easily checks that, for all $k,l\in \NN$, 
\begin{equation}\label{surjective}
dz\wedge dz^* z^kz^{*l} = d\left(-\frac{q^k}{[l+1]_q} dz z^kz^{*l+1}\right).
\end{equation}
If  $z^kz^{*l}\in \cO(C^N_{q,\kappa})$, i.e.\hspace{3pt}$k-l = mN$, for some $m\in \ZZ$, then  $k-l -1 = (m-1)N +N-1$, so that $\deg(z^kz^{*l+1}) = N-1$ and hence the argument of $d$ on the right hand side of \eqref{surjective} is a holomorphic form on $\cO(C^N_{q,\kappa})$. Therefore the map $d: \Omega^1 (C^N_{q,\kappa})\to \Omega^2 (C^N_{q,\kappa})$ is surjective. Since $\nabla = - \pi_{dz\wedge dz^*}\circ d\circ \Theta^{-1}$, where $\Theta^{-1}$ is the isomorphism constructed through Lemma~\ref{lemma.Poincare}, it follows that $\nabla$ is also surjective. Consequently, the integral associated to $\nabla$ vanishes everywhere on $\cO(C^N_{q,\kappa})$.

The divergence $\nabla$ can be computed using the explicit form of the isomorphism $\Theta^{-1}$ constructed in the proof of Lemma~\ref{lemma.Poincare}. As the first step, let us define the operations $\pz, \pzs: \cO(D_{q,\kappa}) \to \cO(D_{q,\kappa})$ by
$$
d\alpha = \pz(\alpha) dz + \pzs(\alpha)dz^* = dz \sigma(\pz(\alpha)) +dz^* \sigma(\pzs(\alpha)),
$$
where $\sigma$ is given explicitly in \eqref{auto.explicit}. Both $\pz$ and $\pzs$ are twisted derivations, i.e.\hspace{3pt}for all $\alpha, \beta\in \cO(D_{q,\kappa})$,
$$
\pz(\alpha\beta) = \pz(\alpha)\sigma^{-1}(\beta) +\alpha \pz(\beta),
$$
and similarly for $\pzs$. In fact, they are {\em q-derivations}, i.e.
\begin{equation}\label{qder}
\sigma\circ \pz \circ \sigma^{-1} = q\pz, \qquad \sigma\circ \pzs \circ \sigma^{-1} = q^{-1}\pzs .
\end{equation}
If we choose $r_\pm^i, s_\pm^i, \in \cO(C^N_{q,\kappa})$, such that $\deg(r_\pm^i) = \deg(s_\pm^i) = \pm 1\, \mbox{mod}\, N$ and $r_+^ir_-^i = s_-^is_+^i =1$ (summation implicit), whose existence is guaranteed by \cite[Theorem~2.1]{Brz:com}, then the formula \eqref{def.theta} combined with all the above identifications of forms on $\cO(C^N_{q,\kappa})$ and with \eqref{qder} gives
\begin{equation}\label{nabla}
\nabla(\phi) = q \pz (\phi(dzs_-^i)s_+^i) + q^{-1} \pzs (\phi(dz^*r_+^i)r_-^i),
\end{equation}
for all right $\cO(C^N_{q,\kappa})$-module maps $\phi: \Omega^1 (C^N_{q,\kappa}) \to \cO(C^N_{q,\kappa})$. We note in passing that the first order calculus $\Omega^1 (D_{q,\kappa})$ is the calculus associated to the twisted multi-derivation $(\pz,\pzs)$ in the sense of \cite[Section~3]{BrzElK:int}, and, for $N=1$, $\nabla$ in \eqref{nabla} is exactly the divergence associated to such a calculus. For $N\neq 1$, \eqref{nabla} coincides with the induced divergence on the base space of a quantum principal bundle obtained by the method described in \cite[Section~4]{BrzElK:int}.

When $\kappa =0$, relations \eqref{disc} define the quantum plane. In this case, the $\ZZ_N$-grading is not strong (apart from the trivial case $N=1$): there is no possible combination of $z^kz^{*k}$ and $z^{*N-k}z^{N-k}$, $k=1,\dots, N-1$, with coefficients of degree zero that would give the identity element. Furthermore $dzz^*$ and $dz^*z$ are not elements of $\Omega^1 (C^N_{q,0})$ (although their linear combination $dzz^* +q^{-1}dz^*z$ is). Finally, $\Omega^2 (C^N_{q,0})$ is generated by $dz\wedge dz^* a^{N-1}$, $dz\wedge dz^* b$ and $dz\wedge dz^* b^{*}$, which are not free; when $\kappa =0$,  relations \eqref{cone} allow one to write 0 as a linear combination (with non-zero coefficients from $\cO(C^N_{q,0})$) of any two out of $a^{N-1}, b, b^*$. This means that $\Omega^2 (C^N_{q,0})$ cannot be isomorphic to $\cO(C^N_{q,0})$, hence these differential calculi over the quantum Manin cones $C^N_{q,0}$ are not integrable.

\subsection{Spectral triple for the Moyal cone}
In the special case of $q=1$ and $\kappa \not= 0$, $\cO(D_{1,\kappa})$ is the  algebra of the Moyal plane, which could be also studied
from the angle of spectral triples. First, observe that the algebra $\cC^\infty(D_{1,\kappa})$ of the series $\sum a_{m,n} z^mz^{*n}$ with rapidly decreasing coefficients is  a subalgebra of $\Ss(\RR^2)$, the Schwartz functions 
on the plane with the product defined through the oscillating integral, 
$$ (f \ast g) (z) = (\pi \kappa)^{-2} \int_\CC d^2s \int_\CC  d^2t \;  f(z + s) g(z + t) e^{-2 i \kappa \Im(s^* t)} , \;\;\; z \in \CC.$$
 
It is easy to see that the action of the cyclic group $\ZZ_N$ on the space of Schwartz functions:
$$ \sigma(f)(z) := f \left( e^{\frac{2}{N} \pi i} z \right), $$ 
is an automorphism of this algebra. Let us consider the invariant subalgebra $\Ss(\RR^2)^{\ZZ_N}$ as the algebra of smooth
functions on the Moyal cone. We can now take the spectral triple for the Moyal plane, as constructed in \cite{Moyal}. 
Using the same procedure as in the case of the noncommutative pillow, we lift the action of $\ZZ_N$ to the Hilbert 
space, the latter being just $L^2(\RR^2) \otimes \CC^2$. The spectral triple over the Moyal plane is constructed 
with the diagonal action of the algebra by left Moyal-multiplication and the Dirac operator and chirality operator
of the form
$$ D = \left( \begin{array}{lr} 0 & \partial_z \\ \partial_{z^*} & 0 \end{array} \right) ,\;\;\;\;
     \gamma =  \left( \begin{array}{lr} 1 & 0 \\ 0 & -1 \end{array} \right).
$$     
Since we can decompose the Hilbert space into a direct sum of spaces on which the group $\ZZ_N$ acts by multiplication by
different roots of unity, it is easy to construct a subspace, which is preserved both by the Dirac operator and  by the action
of the invariant subalgebra. The restriction of the original spectral triple to that subspace gives a spectral triple for the 
Moyal cone.

If we try to consider a nondegenerate (Hochschild) cycle for the Moyal cone we need to observe first that its
construction for the Moyal plane is, in fact, based on the Weyl algebra of plane waves in the multiplier of the Moyal algebra. 
However, since there are no plane waves which are in the multiplier of the invariant subalgebra of the Moyal cone one
cannot use such a construction. 

On the other hand we are dealing here with the case of spectral triples over a locally compact noncommutative space, so we might relax the conditions
and demand that the cycle exists in the algebra of unbounded multiplier of the invariant algebra. Although to show the
existence for any $N$ appears to be rather a challenge, we shall demonstrate in the case of $N=2$ that this is indeed 
possible.

\begin{lemma}
Let $N=2$ and let $\Ss(\RR^2)^{\ZZ_2}$ be the algebra of the Moyal cone with the spectral triple as above. Then the
following cycle gives the volume form $\gamma$:

\begin{eqnarray*}
&& - [D, b^*]  a [D,b]  - 2 a [D,b] [D,b^*] + \frac{3}{2} a [D,b^*]  [D,b]  \\
 && \phantom{xxxxxxxxx} +3b [D,a] [D, b^*] - b [D,b^*] [D,a] + b^* [D,b] [D,a] 
= 4 \kappa^2 \gamma. 
\end{eqnarray*}
\end{lemma}
The proof is by explicit computation and making the Ansatz that the cycle is at most linear in each entry in the generators $a,b,b^*$
of the cone algebra $\cO(C^2_{q,0})$.

Let us observe that again,  as in the case of the noncommutative pillow, the above cycle has a nonvanishing image
only if $\kappa \not= 0$ and that again the solution is not unique. For this particular Ansatz there are no Hochschild cycles 
which could give $\gamma$.

We believe that the construction of such cycles is possible in the noncommutative case for every Moyal cone, it is again an 
open and challenging problem to prove it for every $N$ and to determine whether it is possible to find a Hochschild cycle, which 
gives the orientation.

\section{Three-dimensional integrable differential calculi over quantum lens spaces}
\setcounter{equation}{0}
 The aim of this section is to prove differential smoothness of quantum lens spaces by constructing quantum principal bundles over them with integrable differential calculi that satisfy the assumptions of Lemma~\ref{lem.princ}.

The coordinate algebra $\cO(L_q(N;1,N))$  of the quantum lens space $L_q(N;1,N)$ is a $*$-algebra generated by $\xi,\zeta$ subject to the relations
\begin{subequations} \label{rel.lens}
\begin{gather}
\xi\zeta = q^{l}\zeta\xi, \qquad \xi\zeta^* = q^{l}\zeta^*\xi, \qquad \zeta\zeta^* = \zeta^*\zeta, \label{rel.lens1}\\
\xi\xi^* = \prod_{m=0}^{l-1}(1-q^{2m}\zeta\zeta^*), \qquad \xi^*\xi = \prod_{m=1}^l(1-q^{-2m}\zeta\zeta^*), \label{rel.lens2}
\end{gather}
\end{subequations}
where $q\in(0,1)$; see \cite{HonSzy:len}. In the classical limit $q=1$, this is the coordinate algebra of the singular lens space. A linear basis of the space $\Vv(n)$ spanned by words in generators $\xi,\xi^*,\zeta,\zeta^*$ of length at most $n$ is given by
\begin{equation}\label{basis}
\{\xi^i\zeta^j\zeta^{*k}, \, \xi^{*r+1}\zeta^s\zeta^{*t} \; |\; i+j+k \leq n ,\, r+s+ t \leq n-1, \},
\end{equation}
hence
$$
\dim \Vv(n) = {n+3 \choose 3} +  {n+2 \choose n-1} = \frac 16 (n+1)(n+2)(2n +3)
$$
Consequently, 
$$
\gk (\cO(L_q(N;1,N))) =3.
$$

The algebra $\cO(L_q(N;1,N))$ embeds as a $*$-algebra into $\cO(SU_q(2))$, the coordinate algebra of the quantum group $SU_q(2)$.  $\cO(SU_q(2))$ is a $*$-algebra generated by $\alpha$ and $\beta$ subject to the following relations
\begin{subequations}\label{su}
\begin{gather}
\alpha\beta = q\beta\alpha, \qquad \alpha\beta^* = q\beta^*\alpha, \qquad \beta\beta^* = \beta^*\beta, \label{su.a}\\
 \alpha\alpha^* = \alpha^*\alpha +(q^{-2}-1)\beta\beta^*, \qquad
\alpha\alpha^* + \beta\beta^* =1, \label{su.b}
\end{gather}
\end{subequations}
where $q\in(0,1)$; see \cite{Wor:com}. The embedding $\cO(L_q(N;1,N))\hookrightarrow \cO(SU_q(2))$ is 
\begin{equation}\label{linsu}
\xi \mapsto \alpha^N \quad \mbox{and} \quad  \zeta \mapsto  \beta. 
\end{equation}
Henceforth we view $\cO(L_q(N;1,N))$ as a subalgebra of $\cO(SU_q(2))$ via the embedding \eqref{linsu}. Next we  construct a differential calculus on $\cO(L_q(N;1,N))$ as the restriction of the left covariant three-dimensional calculus  $\Omega (SU_q(2))$ of Woronowicz \cite{Wor:twi} over $\cO(SU_q(2))$. The calculus $\Omega (SU_q(2))$ is a connected $*$-calculus generated by the one-forms $\omega_0$, $\omega_\pm$, which satisfy the following relations:
\begin{subequations}\label{eq.3Drel}
\begin{gather}
\omega_0\,\alpha \,=\, q^{-2}\,\alpha\,\omega_0\, , \qquad 
\omega_0\,\beta \,=\, q^2\,\beta\,\omega_0\hspace{3pt},\\
\omega_\pm\,\alpha \,=\, q^{-1}\,\alpha\,\omega_\pm\, , \qquad
\omega_\pm\,\beta \,=\, q\,\beta\,\omega_\pm\hspace{3pt}, \\
\omega_i\wedge \omega_i  =  0, \quad 
 \omega_+\wedge \omega_-=-q^2\,\omega_-\wedge \omega_+\, ,\quad
\omega_\pm \wedge \omega_0\,=\,-q^{\pm 4}\,\omega_0\wedge \omega_\pm\, ,
\end{gather}
\end{subequations}
and their $*$-conjugates with $\omega_0^* = -\omega_0$, $\omega_-^* = q\omega_+$. The differential is given by
\begin{subequations}\label{eq.3dif}
\begin{gather}
d \alpha \,=\, \alpha\,\omega_0-q\,\beta\,\omega_+\, ,\qquad
d \beta \,=\, -q^2\,\beta\,\omega_0 + \alpha\,\omega_-\, , \label{eq.3dif.a}\\
d \omega_0 =q\omega_-\wedge \omega_+\, ,\qquad
d \omega_+ =q^2(q^2+1)\,\omega_0\wedge \omega_+\,  . \label{eq.3dif.b}
\end{gather}
\end{subequations}
This is a calculus over $\cO(SU_q(2))$, since
$$
\omega_0 = \alpha^*d\alpha + q^{-2}\beta d\beta^*, \qquad \omega_+ = q^{-2}\alpha d\beta^* - q^{-1}\beta^*d\alpha,
$$
by \eqref{eq.3dif.a} and \eqref{su}.

\begin{theorem}\label{thm.lens}
For all values of $q\in(0,1)$ and $N\in \NN$, the restriction $\Omega (L_q(N;1,N))$ of the 3D-calculus $\Omega (SU_q(2))$ to  $\cO(L_q(N;1,N))$ is a three-dimensional integrable differential calculus. Consequently, $\cO(L_q(N;1,N))$ is a differentially smooth algebra. 
\end{theorem}
\begin{proof}
As explained in \cite[Theorem~2]{Brz:smo}, $\cO(SU_q(2))$ can be made into a  principal $\CC\ZZ_N$-comodule algebra with coinvariants equal to $\cO(L_q(N;1,N))$. The Hopf $*$-algebra  $\CC\ZZ_N$ is generated by a unitary group-like element $u$ such that $u^N=1$. 
The $\CC\ZZ_N$-coaction is a $*$-algebra map given on the generators of $\cO(SU_q(2))$ by
\begin{equation}\label{coac.zn}
\alpha \mapsto \alpha \ot u, \qquad \beta \mapsto \beta \ot 1.
\end{equation}
The coinvariant subalgebra is generated by $\beta$, $\alpha^N$ and thus can be identified with $\cO(L_q(N;1,N))$ by \eqref{linsu}. 

The calculus $\Omega (SU_q(2))$ is a $\CC\ZZ_N$-covariant calculus by the $*$-coaction
\begin{equation}\label{cov.coa}
\omega_0\mapsto \omega_0\ot 1, \qquad \omega_\pm\mapsto \omega_\pm\ot u^{\pm 1}.
\end{equation}
It is shown in \cite[Section~4.1]{BrzElK:int} that  $\Omega (SU_q(2))$   is  integrable
with an integrating volume form $\omega = \omega_-\wedge\omega_0\wedge \omega_+$.
Thus we are in a situation to which Lemma~\ref{lem.princ} can be applied provided assumptions (a)--(c) are satisfied.

First we look at the coinvariant part of $\Omega (SU_q(2))$ and study it degree by degree. The coinvariant part of $\Omega^1 (SU_q(2))$ can be identified with
$$
\langle \omega_0,\hspace{3pt}\alpha^*\omega_+,\hspace{3pt}\alpha^{N-1}\omega_+,\hspace{3pt}\alpha\omega_-,\hspace{3pt}\alpha^{*N-1}\omega_-\rangle \cO(L_q(N;1,N)).
$$
Thus is suffices to show that $\omega_0, \alpha^*\omega_+,\alpha^{N-1}\omega_+,\alpha\omega_-, \alpha^{*N-1}\omega_- \in \Omega^1 (L_q(N;1,N))$, i.e.\hspace{3pt}that each of these forms can be expressed as linear combinations of $bdb', db\, b'$, $b,b'\in  \cO(L_q(N;1,N))$.  Starting with the second of the relations \eqref{eq.3dif.a} and its $*$-conjugate, and using the commutation rules \eqref{eq.3Drel}, \eqref{su} we find
$$
\beta d\beta^* = \beta\beta^*\,\omega_0 + q^2 \beta\alpha^*\,\omega_+, \qquad d\beta^*\beta  = q^2\beta\beta^*\,\omega_0 + q^2 \beta\alpha^*\,\omega_+.
$$
Hence
$$
\beta\alpha^*\omega_+ = \frac{1}{q^2-1}\left( \beta d\beta^* - q^{-2}(d\beta^*)\beta\right) \in \Omega^1 (L_q(N;1,N)).
$$
Consequently, also $\beta\beta^*\omega_0  \in \Omega^1 (L_q(N;1,N))$.  The first of relations \eqref{eq.3dif.a} combined with the commutation rules \eqref{eq.3Drel}, \eqref{su} yields
\begin{equation}\label{dan}
d\alpha^N = [N]_{q^{-2}}\left(\alpha^N\omega_0 - q\alpha^{N-1}\beta \omega_+\right),
\end{equation}
where $[N]_{q^{-2}}$ denotes the $q$-integer \eqref{q.integers}. By \eqref{su.b}, $\alpha^{*k}\alpha^k$ is a polynomial in $\beta\beta^*$ with the constant term 1, hence \eqref{dan} gives
\begin{equation}\label{omega0}
[N]_{q^{-2}}\omega_0   = - \alpha^{*N}d\alpha^N + f_1(\beta\beta^*)\beta\beta^*\omega_0 + f_2(\beta\beta^*)\beta\alpha^* \omega_+,
\end{equation}
for some polynomials $f_1$ and $f_2$. Hence  $\omega_0  \in \Omega^1 (L_q(N;1,N))$. The $*$-conjugate of the second of equations \eqref{eq.3dif.a} implies that $\alpha^*\omega_+ \in \Omega^1 (L_q(N;1,N))$ too.

Again using the $*$-conjugate of the second of  equations \eqref{eq.3dif.a} as well as relations \eqref{eq.3Drel}, \eqref{su}  we find
$$
\alpha^{N-1}\omega_+ = \frac{q^{-2}}{q^{2N}-1} \left( q^{3N} (d\beta^*)\alpha^N - \alpha^N d\beta^*\right) \in \Omega^1 (L_q(N;1,N)).
$$
Taking suitable $*$-conjugates we conclude, therefore, that all the generating forms of $\Omega^1 (SU_q(2))^{co \CC\ZZ_N}$ are elements of $\Omega^1 (L_q(N;1,N))$, i.e.\
$$
\Omega^1 (SU_q(2))^{co \CC\ZZ_N}= \Omega^1 (L_q(N;1,N)),
$$
as required. 

Next, studying two-forms one finds that the coinvariant part of $\Omega^2 (SU_q(2))$ is 
$$
\langle \omega_-\wedge \omega_+ ,\, \alpha^*\omega_+\wedge \omega_0,\hspace{3pt}\alpha^{N-1}\omega_+\wedge \omega_0,\hspace{3pt}\alpha\omega_-\wedge \omega_0,\hspace{3pt}\alpha^{*N-1}\omega_-\wedge \omega_0\rangle \cO(L_q(N;1,N)).
$$
All but the first of these generating two-forms are obtained as products of one-forms which have already been shown to be in $\Omega^1 (L_q(N;1,N))$; thus they are elements of $\Omega^2 (L_q(N;1,N))$. Using relations \eqref{su.b} and \eqref{eq.3Drel} one easily finds that
$$
\omega_-\wedge \omega_+ = \frac{1}{1-q^2}\left(q^{-1}\alpha\omega_-\wedge \alpha^* \omega_+ + q^5\alpha^* \omega_+ \wedge \alpha\omega_-\right).
$$
Hence also $\omega_-\wedge \omega_+ $ is a linear combination of products of one-forms in $\Omega^1 (L_q(N;1,N))$, thus it is in $\Omega^2 (L_q(N;1,N))$. This proves that
$$
\Omega^2 (SU_q(2))^{co \CC\ZZ_N}= \Omega^2 (L_q(N;1,N)).
$$
Finally, the volume form $\omega = \omega_-\wedge\omega_0\wedge \omega_+$ is the product of forms in $\Omega (L_q(N;1,N))$, thus both the coinvariant part of $\Omega^3 (SU_q(2))$ equals $\Omega^3 (L_q(N;1,N))$ and $\omega$ is coinvariant. Therefore, assumptions (a)--(b) of Lemma~\ref{lem.princ} are satisfied. 

To check  assumption (c) it might be convenient to interpret the $\CC\ZZ_N$-coaction in terms of the $\ZZ_N$-grading. In that way $\Omega (SU_q(2))$ is a $\ZZ_N$-graded algebra with the $*$-compatible grading  given on generators by
$$
\deg(\alpha) = \deg(\omega_+) =1, \qquad \deg(\beta)= \deg(\omega_0) =0.
$$
As is shown above, $\Omega (L_q(N;1,N))$ is the degree-zero subalgebra of $\Omega (SU_q(2))$. If we take any $\omega' = a_-\omega_- +a_0\omega_0 + a_+\omega_+ \in \Omega^1 (SU_q(2))$ and assume that, for all $\omega'' \in \Omega^2 (L_q(N;1,N))$, $\omega'\wedge \omega'' \in \Omega^3 (L_q(N;1,N))$, i.e.\hspace{3pt}$\deg(\omega'\wedge \omega'') =0$, then in particular, taking $\omega'' = \omega_-\wedge \omega_+$ we find
$$
0 = \deg( \omega'\wedge \omega_-\wedge \omega_+) = \deg (a_0 \omega_0\wedge \omega_-\wedge \omega_+) = \deg(a_0).
$$
Similarly, by taking $\omega'' = \alpha^* \omega_0\wedge \omega_+$ and $\omega'' = \alpha \omega_-\wedge \omega_0$, one finds that $\deg(a_\mp) = \pm1$. Therefore, $\deg(\omega') =0$, as required.

In a similar way, assuming that $\omega' = a_-\omega_0\wedge \omega_+ +a_0\omega_-\wedge\omega_+ + a_+\omega_-\wedge\omega_0 \in \Omega^2 (SU_q(2))$ is such that for all $\omega'' \in \Omega^1 (L_q(N;1,N))$, $\omega'\wedge \omega'' \in \Omega^3 (L_q(N;1,N))$, and choosing $\omega''$ to be $\omega_0$, $ \alpha^*  \omega_+$ and $\alpha \omega_-$, one finds that $\deg(a_0) =0$ and $\deg(a_\pm) = \pm 1$. Therefore, $\deg(\omega') =0$, as needed. This proves that  assumption (c) of Lemma~\ref{lem.princ} is also satisfied, and consequently that $\Omega (L_q(N;1,N))$  is a  three-dimensional integrable calculus, as needed.
\end{proof}

A divergence on $\cO(SU_q(2))$ corresponding to the 3D-calculus was derived in \cite[Section~4.1]{BrzElK:int}. The corresponding divergence on $\cO(L_q(N;1,N))$, computed from the formula \eqref{nablas} comes out as, for all $\phi\in \Ii_1 (L_q(N;1,N))$,
$$
\nabla  (\phi) = \partial_0\left( \hat{\phi}\left(\omega_0\right)\right) + q^{-2}\partial_+\left( \hat{\phi}\left(\omega_+\right)\right) + q^{2}\partial_-\left( \hat{\phi}\left(\omega_-\right)\right),
$$
where $\partial_i$ are defined by $da = \partial_i(a)\omega_i$. The maps $\hat{\phi}$  are defined by \eqref{hat.phi} and can be computed in terms of $\phi$:
$$
\hat{\phi} (\omega_0) = \phi(\omega_0), \qquad \hat{\phi}(\omega_+) = x_0\phi\left(\omega_+\alpha^{N-1}\right) \alpha^{*N-1}+ \sum_{i=1}^{N-1}x_i\phi\left(\omega_+\alpha^{*}\right)\alpha (\beta\beta^*)^i,
$$
$$
 \hat{\phi}(\omega_-) = y_0\phi\left(\omega_-\alpha^{*N-1}\right) \alpha^{N-1}+ \sum_{i=1}^{N-1}y_i\phi\left(\omega_-\alpha\right)\alpha ^{*}(\beta\beta^*)^i,
$$
where $x_i, y_i\in \CC$ are solutions of 
$$
x_0\alpha^{N-1} \alpha^{*N-1}+ \sum_{i=1}^{N-1}x_i\alpha^{*}\alpha (\beta\beta^*)^i = 1, \qquad y_0\alpha^{*N-1}\alpha^{N-1}+ \sum_{i=1}^{N-1} y_i\alpha\alpha ^{*}(\beta\beta^*)^i =1\, .
$$
The existence of such $x_i, y_i$ is proven in \cite[Lemma~2]{Brz:smo}. 

Finally, it is shown in \cite[Section~4.1]{BrzElK:int} that the integral associated to the 3D-calculus coincides with a scalar multiple of the normalized integral on the Hopf algebra $\cO(SU_q(2))$ or the invariant Haar integral on $SU_q(2)$ described in \cite[Appendix~1]{Wor:com}. The formula \eqref{integrals} implies that the integral on $\cO(L_q(N;1,N))$ is the restriction of the Haar integral on $SU_q(2)$, given by 
\begin{equation}\label{integral}
\Lambda\left(\zeta^k\zeta^{*k}\right) =  \frac{1}{~{}~{}~{}~[k]_{q^{-2}}}, 
\end{equation}
and zero on all other elements $\xi^i\zeta^j\zeta^{*k}$,  $\xi^{*i+1}\zeta^j\zeta^{*k}$, $i,j,k,\in \NN$, $j\neq k$ of the linear basis \eqref{basis} for $\cO(L_q(N;1,N))$.

\section*{Acknowledgements} 
The work on this project started during two visits of the first author to the Institute of Physics, Jagiellonian University; he would like to thank the members of the Institute for their hospitality.

\end{document}